\documentclass[12pt,leqno]{amsart}

\usepackage{amssymb}
\usepackage{tikz}
\usepackage{tikz-cd}
\usepackage{enumerate}

\newtheorem{theorem}{Theorem}[section]
\newtheorem{proposition}[theorem]{Proposition}
\newtheorem{lemma}[theorem]{Lemma}
\newtheorem{corollary}[theorem]{Corollary}

\theoremstyle{definition}

\newtheorem{definition}[theorem]{Definition}

\newtheorem{example}[theorem]{Example}

\newtheorem{remark}[theorem]{Remark}

\newcommand{\N}{{\mathbb{N}}}
\newcommand{\R}{{\mathbb{R}}}

\newcommand{\Z}{{\mathbb{Z}}}

\newcommand{\Cinf}{C^\infty}

\newcommand{\cor}{{\bf k}}

\renewcommand{\to}[1][]{\xrightarrow[]{#1}}
\newcommand{\from}[1][]{\xleftarrow[]{#1}}
\newcommand{\isoto}[1][]{\xrightarrow[#1]%
{{\raisebox{-.6ex}[0ex][-.6ex]{$\mspace{1mu}\sim\mspace{2mu}$}}}}
\newcommand{\isofrom}[1][]{\xleftarrow[#1]%
{{\raisebox{-.6ex}[0ex][-.6ex]{$\mspace{1mu}\sim\mspace{2mu}$}}}}

\newcommand{\RR}{\mathrm{R}}

\newcommand{\Hom}{\mathrm{Hom}}

\newcommand{\Ext}{\mathrm{Ext}}
\renewcommand{\hom}{{\mathcal{H}om}}
\newcommand{\rhom}{{\RR\hom}}
\newcommand{\DD}{\mathrm{D}}
\newcommand{\tens}{\otimes}
\newcommand{\ltens}{\mathbin{\overset{\scriptscriptstyle\mathrm{L}}\tens}}

\newcommand{\sect}{\Gamma}
\newcommand{\rsect}{\mathrm{R}\Gamma}

\newcommand{\roim}[1]{\RR{#1}_*}
\newcommand{\reim}[1]{\RR{#1}_!}
\newcommand{\opb}[1]{#1^{-1}}
\newcommand{\epb}[1]{#1^{!}}

\newcommand{\eqdot}{\mathbin{:=}}
\newcommand{\cl}{\colon}
\newcommand{\pointdiag}{\makebox[0mm]{\;\;.}}
\newcommand{\virgdiag}{\makebox[0mm]{\;\;,}}
\newcommand{\scbul}{{\,\raise.4ex\hbox{$\scriptscriptstyle\bullet$}\,}}
\newcommand{\ol}{\overline}

\newcommand{\id}{\mathrm{id}}

\newcommand{\supp}{\operatorname{supp}}

\newcommand{\SSi}{\mathrm{SS}}
\newcommand{\Int}{\operatorname{Int}}

\newcommand{\Der}{\mathsf{D}}
\newcommand{\Derb}{\Der^{\mathrm{b}}}
\newcommand{\Derlb}{\Der^{\mathrm{lb}}}

\newcommand{\Mod}{\operatorname{Mod}}
\newcommand{\coker}{\operatorname{coker}}

\newcommand{\dT}{{\dot{T}}}

\newcommand{\tw}[1]{\widetilde{#1}}

\numberwithin{equation}{section}
\usepackage{hyperref}

\newcommand{\Conv}{\mathrm{Conv}}
\newcommand{\mo}{\mathopen}

\begin{document}

\begin{abstract}
The Gromov-Eliashberg theorem says that the group of symplectomorphisms of a
symplectic manifold is $C^0$-closed in the group of diffeomorphisms. This can
be translated into a statement about the Lagrangian submanifolds which are
graphs of symplectomorphisms.  It is also known that such Lagrangian
submanifolds are locally microsupports of sheaves. We explain how we can deduce
the Gromov-Eliashberg theorem from the involutivity theorem of Kashiwara and
Schapira which says that the microsupport of a sheaf is coisotropic.
\end{abstract}

\date{October 31, 2013}
\title{The Gromov-Eliashberg theorem by microlocal sheaf theory}
\author{St{\'e}phane Guillermou}
\renewcommand{\shorttitle}{Gromov-Eliashberg by microlocal sheaf theory}

\maketitle


\setcounter{tocdepth}{1}
\tableofcontents

\section{Introduction}

In~\cite{T08}, D.~Tamarkin gives a totally new approach for treating questions
in symplectic geometry (especially classical problems of non-displaceability).
His approach is based on the microlocal theory of sheaves, introduced and
developed in \cite{KS82,KS85,KS90}.  In particular he remarks that in some
situations it is possible to associate a sheaf with a given Lagrangian
submanifold of a cotangent bundle and then deduce properties of the Lagrangian
submanifold from this sheaf.  In this paper we follow this approach and explain
how we can recover the Gromov-Eliashberg theorem by microlocal sheaf theory,
using in particular the involutivity of the microsupport.

\smallskip

Let us briefly recall some facts of the microlocal theory of sheaves. We
consider a real manifold $M$ of class $C^\infty$ and a commutative unital ring
$\cor$ of finite global dimension. We let $\Derb(\cor_M)$ be the bounded derived
category of sheaves of $\cor$-modules on $M$.  In~\cite{KS90}, the authors
attach to an object $F$ of $\Derb(\cor_M)$ its microsupport, or singular
support, $\SSi(F)$, a subset of $T^*M$, the cotangent bundle of $M$. By
definition the microsupport is closed and conic for the action of $\R^+$ on
$T^*M$. A deep result of~\cite{KS90} says that $\SSi(F)$ is involutive (or
coisotropic). The initial motivation for this theorem comes from the theory of
systems of linear PDE's because of its link with the propagation of
singularities.  (A microdifferentiel version of the involutivity theorem is
given in~\cite{SKK} and an algebraic statement is given in~\cite{Ga81}.)

In~\cite{GKS12} the authors prove the following result, inspired by~\cite{T08}.
Let $M$ be a manifold and set $\dT^*M = T^*M \setminus M$.  For $F \in
\Derb(\cor_M)$ we also set $\dot\SSi(F) = \SSi(F) \cap \dT^*M$.  Let $I =
\mo{]}a,b[$ be an interval containing $0$ and let $\psi \cl \dT^*M \times I \to
\dT^*M$ be a homogeneous Hamiltonian isotopy. For $t\in I$ we let $\psi_t$ be
the restriction of $\psi$ at time $t$ and we denote by $\Lambda_{\psi_t} \subset
\dT^*M^2$ the graph of $\psi_t$, twisted by the antipodal map. Hence
$\Lambda_{\psi_t}$ is a conic Lagrangian submanifold.  Then the main result
of~\cite{GKS12} says that there exists $K_t\in \Der(\cor_{M^2})$, for each
$t\in I$, such that $\dot\SSi(K_t) = \Lambda_{\psi_t}$.  We can also consider
non homogeneous Hamiltonian isotopies by adding a variable: given a Hamiltonian
isotopy $\varphi$ of $T^*M$, with compact support, we can define a homogeneous
Hamiltonian isotopy $\psi$ of $\dT^*(M\times\R)$ which makes a commutative
diagram with $\varphi$ and the map
\begin{equation}\label{eq:intro-def-rho}
\rho_M \cl T^*M\times\dT^*\R \to T^*M, \qquad
(x,s;\xi,\sigma) \mapsto (x;\xi/\sigma).  
\end{equation}

The Gromov-Eliashberg theorem (Theorem~\ref{thm:GE} below) says that, if a
sequence of symplectic $C^1$ diffeomorphisms $\{\varphi_n\}_{n\in\N}$ of some
symplectic manifold $(X,\omega)$ has a $C^0$ limit, says $\varphi_\infty$, and
$\varphi_\infty$ is a $C^1$ diffeomorphism of $X$, then $\varphi_\infty$ is
symplectic.  The aim of this paper is to explain how it can be deduced from the
involutivity theorem of~\cite{KS90}.

The Gromov-Eliashberg theorem is in fact a local statement and we can assume
that $X=\R^{2n}$, that $\varphi_n$ is the time $1$ of a Hamiltonian isotopy and
that the convergence occurs on some ball $B$ of $\R^{2n}$.  We identify
$\R^{2n}$ with $T^*\R^n$ and we add a variable to make the situation
homogeneous. Then we can apply the results of~\cite{GKS12} and we deduce that
there exists $K_n \in \Derb(\cor_{\R^{2n+1}})$ such that
$\dot\SSi(K_n) \subset T^*\R^{2n}\times\dT^*\R$ and
\begin{equation}\label{eq:intro-exist-quant}
\rho_{\R^{2n}} (\dot\SSi(K_n) )  = \Lambda_{\varphi_n},
\end{equation}
where $\rho_{\R^{2n}}$ is defined in~\eqref{eq:intro-def-rho} and
$\Lambda_{\varphi_n} \subset T^*\R^{2n}$ is the twisted graph of
$\varphi_n$. We define $K$ by the distinguished triangle
$\bigoplus_{n\in\N} K_n \to \prod_{n\in\N} K_n \to K \to[+1]$.
Then $\SSi(K) \subset \bigcap_{k\in \N^*} \ol{\bigcup_{n\geq k} \SSi(K_n)}$ and
we have in particular
\begin{equation}\label{eq:incl_SS_Lambda}
\rho_{\R^{2n}} (\dot\SSi(K)) \subset \Lambda_{\varphi_\infty}.
\end{equation}
Using this inclusion we can deduce from the involutivity theorem that
$\Lambda_{\varphi_\infty}$ is coisotropic at any point $p$ which belongs to
$\rho_{\R^{2n}} (\SSi(K))$.  This means that it only remains to prove
that~\eqref{eq:incl_SS_Lambda} is in fact an equality.  We do not prove
it directly. We only prove that, for any given $p\in \Lambda_{\varphi_\infty}$, we
can modify the $K_n$'s by a so called cut-off result of~\cite{KS90} and obtain
another sheaf $K$ (depending on $p$) such that~\eqref{eq:incl_SS_Lambda}
still holds and moreover $p\in \rho_{\R^{2n}} (\SSi(K))$.
Then the involutivity theorem applies at $p$. We obtain in this way that
$\Lambda_{\varphi_\infty}$ is coisotropic at all points. Since it is a
submanifold of dimension $2n$, it is Lagrangian, which means that
$\varphi_\infty$ is a symplectic map.

\bigskip

Now we give a more precise idea of the proof.
We first state the result in the following local form.
Let $(E,\omega)$ be a symplectic vector space which we identify with $\R^{2n}$.
We endow $E$ with the Euclidean norm of $\R^{2n}$.  For $R>0$ we let $B_R^E$ be
the open ball of radius $R$ and center $0$.  For a map $\psi\cl B_R^E \to E$ we
set $\| \psi \|_{B_R^E} = \sup\{ \| \psi(x) \|;$ $x\in B_R^E \}$.

\begin{theorem}[Gromov-Eliashberg rigidity theorem, see~\cite{E87, Gr86}]\label{thm:GE}
Let $R>0$. Let $\varphi_n\cl B_R^E\to E$, $n\in \N$, and
$\varphi_\infty\cl B_R^E \to E$ be $C^1$ maps.  We assume
\begin{itemize}
\item [(i)] $\varphi_n$ is a symplectic map, that is,
  $\varphi_n^*(\omega) = \omega$, for all $n\in \N$,
\item [(ii)] $\| \varphi_n - \varphi_\infty \|_{B_R^E} \to 0 $ when
  $n\to \infty$,
\item [(iii)] $d \varphi_{\infty,x} \cl T_xE \to T_{\varphi_\infty(x)}E$ is an
  isomorphism, for all $x\in B_R^E$.
\end{itemize}
Then $\varphi_\infty|_{B_R^E}$ is a symplectic map.
\end{theorem}

\subsection*{Main ingredients of the proof}
As seen above the essential ingredient of the proof is the involutivity theorem
of~\cite{KS90}.  The second ingredient is the main result of~\cite{GKS12} which
implies the existence of a ``quantization'' $K_n$ for $\varphi_n$, that is, an
object $K_n \in \Derb(\cor_{\R^{2n+1}})$ satisfying~\eqref{eq:intro-exist-quant}.

The third important tool is a ``cut-off'' result of~\cite{KS90}.
We use the following statement. Let $V'$ be a vector space and $V=V'\times\R$,
with coordinates $(x_1,\ldots,x_n)$.
Let $\gamma_{c_2} \subset \gamma_{c_1} \subset V$ be two closed cones of the type
$\gamma_c = \{x_n \geq c (x_1^2+\cdots +x_{n-1}^2)^{1/2}\}$, with $c_2>c_1>0$.
Let $B^V_{R}\subset V$ be the open ball of center $0$ and radius $R$.
Let $F\in \Derb(\cor_{B^V_R})$ with a microsupport contained in the union of the
polar cone $\gamma_{c_1}^{\circ a}$ and the complement of $\gamma_{c_2}^{\circ a}$:
\begin{equation}\label{eq:hypSSF_splitBIS}
\dot\SSi(F) \cap \bigr( B^V_R \times 
(\gamma_{c_2}^{\circ a} \setminus \Int(\gamma_{c_1}^{\circ a})) \bigl) = \emptyset.
\end{equation}
The cut-off lemma says that we can decompose $F$ according to this decomposition
of $\SSi(F)$.  More precisely, there exists $r$ such that $R>r>0$ and the
following holds.
For any $F\in \Derb(\cor_{B^V_R})$ satisfying~\eqref{eq:hypSSF_splitBIS}
there exists a distinguished triangle over the smaller ball $B^V_r$,
$F_1 \oplus F_2 \to F|_{B^V_r} \to L \to[+1]$,
such that $\dot\SSi(L) = \emptyset$,
$\dot\SSi(F_1) = \dot\SSi(F) \cap (B^V_r \times \gamma_{c_1}^{\circ a})$
and $\dot\SSi(F_2) \cap (B^V_r \times \gamma_{c_1}^{\circ a}) = \emptyset$.

We can use this decomposition to analyze $F$ and obtain some consequences on its
cohomology. In particular we prove the following result (see
Proposition~\ref{prop:sections-fixed-open}).  We use a notion of convex hull
$\Conv(S)$ for a subset $S\subset T^*M$ of a cotangent bundle: it is the union
of the convex hulls in each fiber, that is,
$\Conv(S) = \bigsqcup_{x\in M} \Conv(S\cap T^*_xM)$.
Given $R, c_1,c_2$ as above, there exist non empty connected open subsets
$W_1\subset\dots\subset W_4$ of $B^V_R \times \R$ such that:
if $F\in \Derb(\cor_{B^V_R})$ satisfies~\eqref{eq:hypSSF_splitBIS},
$\dot\SSi(F)$ is a Lagrangian submanifold, $F$ is simple (see
Section~\ref{sec:simplesheaves}) and the map
$$
(\dot\SSi(F) \cap (B^V_R \times \gamma_{c_1}^{\circ a}))/\R_{>0} \to B^V_R
$$
is proper of degree $1$, then there exists a sheaf $F'\in \Mod(\cor_{W_4})$
such that $\rho_{V'}(\SSi(F')) \subset \Conv(\rho_{V'}(\SSi(F)))$
and the groups $\sect(W_i;F')$, $i=1,\dots,4$, are distinct.

\subsection*{Idea of the proof}
We prove that $T_p\Lambda_{\varphi_\infty}$ is coisotropic, for any given
$p\in \Lambda_{\varphi_\infty}$. We work near $p$ and we approximate $\varphi_n$
by a globally defined Hamiltonian isotopy whose graph coincides with
$\Lambda_{\varphi_n}$ near $p$. Hence we can assume that $\varphi_n$ is the time
$1$ of a Hamiltonian isotopy.  By the main result of~\cite{GKS12} there exists
$K_n \in \Derb(\cor_{V^2 \times \R})$, for each $n$, such that
$\rho_{V^2} (\dot\SSi(K_n)) = \Lambda_{\varphi_n}$.  By the consequence of the
cut-off result explained in the previous paragraph we can find connected open
subsets $W_1 \subset\dots\subset W_4$ of $V^2 \times \R$ and
$L_n \in \Mod(\cor_{V^2\times\R})$
such that $\rho_{V^2} (\dot\SSi(L_n)) \subset \Conv(\Lambda_{\varphi_n})$ near
$p$ and the groups $\sect(W_i;L_n)$, $i=1,\dots,4$, are distinct.

We define $L\in \Mod(\cor_{V^2\times\R})$ by
$L = \coker(\bigoplus_{n\in\N} L_n \to \prod_{n\in\N} L_n)$.
Then $\rho_{V^2} (\SSi(L))$ is contained in the limit of the
$\Conv(\Lambda_{\varphi_n})$. We can assume from the beginning that
$\Lambda_{\varphi_\infty}$ is a section of the projection $T^*V^2 \to V^2$, near
$p$. Then the limit of the $\Conv(\Lambda_{\varphi_n})$ is $\Lambda_{\varphi_\infty}$.

We can see also that the groups $\sect(W_i;L)$, $i=1,\dots,4$, are distinct,
which implies that $L$ has a non trivial microsupport somewhere over
$W_4$. Hence there exists $q=(y;\eta) \in \dot\SSi(L)$ such that $y\in W_4$. By
the involutivity theorem we know that $\dot\SSi(L)$ is coisotropic at $q$. It
follows that $\Lambda_{\varphi_\infty}$ is coisotropic at $p'=\rho_{V^2}(q)$.
Now $W_4$ can be made as small as we want so that $p'$ is arbitrarily close to
$p$. It follows that $\Lambda_{\varphi_\infty}$ is coisotropic at $p$.

The proof is detailled in Section~\ref{sec:GET}. Only
Section~\ref{sec:noncstsections} contains new results. The other sections are
reminders of some notions on sheaves and results of~\cite{KS90}.  The
reader may also consult~\cite{V11} for an introduction to the use of sheaves
theory in symplectic geometry. The paper~\cite{Vic12} gives another application
of microlocal sheaf theory to the study of the $C^0$-rigidity of the Poisson
bracket.

\subsection*{Acknowledgments}
The idea of applying the involutivity theorem to the $C^0$-rigidity emerged
after several discussions with Claude Viterbo, Pierre Schapira and Vincent
Humili\`ere (in particular about the paper~\cite{HLS13}).  It is a pleasure to
thank them for their interest in this question.

\section{Microlocal theory of sheaves}
\label{section:mts}
In this section, we recall some definitions and results from~\cite{KS90},
following its notations with the exception of slight modifications. We consider
a manifold $M$ of class $C^\infty$.

\subsubsection*{Some geometrical notions  (\cite[\S 4.2,~\S 6.2]{KS90})}
For a locally closed subset $A$ of $M$, we denote by $\Int(A)$
its interior and by $\overline{A}$ its closure.

We denote by $\pi_M \cl T^*M\to M$ the cotangent bundle of $M$.  If $N\subset
M$ is a submanifold, we denote by $T^*_NM$ its conormal bundle.
We identify $M$ with $T^*_MM$, the zero-section of $T^*M$.
We set $\dT^*M = T^*M\setminus T^*_MM$ and we denote by
$\dot\pi_M\cl\dT^*M\to M$ the projection. We let $a_M\cl T^*M \to T^*M$ be the
antipodal map $(x;\xi) \mapsto (x;-\xi)$. For a subset $A$ of $T^*M$ we set
$A^a = a_M(A)$.

Let $f\cl M\to N$ be  a morphism of  real manifolds. It induces morphisms
on the cotangent bundles:
$$
T^*M \from[\; f_d\;] M\times_N T^*N \to[\;f_\pi\;] T^*N.
$$
We denote by $\Gamma_f \subset M\times N$ the graph of $f$.
If $\varphi\cl T^*X \to T^*Y$ is a map between cotangent bundles we also
consider the twisted graph
\begin{equation}\label{eq:def_tw_graph}
\Lambda_\varphi =\Gamma_{a_Y\circ \varphi}.
\end{equation}

The cotangent bundle $T^*M$ carries an exact symplectic structure.  We denote
the symplectic form by $\omega_M$.  It is given in local coordinates $(x;\xi)$
by $\omega_M = \sum_i d\xi_i \wedge dx_i$.

For a normed vector space $(E, \| . \|)$, a point $x\in E$ and $r \geq 0$ we
denote by $B_{x,r}^E$ the open ball of radius $r$ and center $0$.  If $x=0$, we
usually write $B_r^E$ for $B_{0,r}^E$.  For an open subset $U \subset E$ and a
continuous map $\psi\cl U \to E$ we set
\begin{align}
\| \psi \|_U &= \sup \{ \| \psi(x) \|;\; x\in U \}, \\
\label{eq:norm1}
\| \psi \|^1_U &= \sup \{ \| \psi(x) \|,\, \| d\psi_x(v)\|;\;
x\in U,\, \|v\| =1 \}, \text{ if $\psi$ is $C^1$.}
\end{align}

A subset of the cotangent bundle $T^*M$ is called $\R^+$-conic (or conic) if it
is invariant by the action of $(\R^+,\times)$ on $T^*M$.  We can turn non conic
subsets into conic ones by adding a variable and taking the inverse image by the
following map $\rho_M$. Let $(s;\sigma)$ be the coordinates on $T^*\R$. 
We define $\rho_M \cl T^*M\times\dT^*\R \to T^*M$ by
\begin{equation}\label{eq:def_rho}
\rho_M(x,s;\xi,\sigma) = (x;\xi/\sigma).  
\end{equation}
Finally we set $T^*_{\sigma>0}(M\times\R)
= \{ (x,s;\xi,\sigma) \in T^*(M\times\R)$; $\sigma>0\}$.

\subsubsection*{Microsupport}
In this paper the coefficient ring $\cor$ is assumed to be a field.  This makes
the description of simple sheaves easier (see Section~\ref{sec:simplesheaves}).
However the theory of microsupport works for a commutative unital ring of finite
global dimension.  We denote by $\Mod(\cor_M)$ the category of sheaves of
$\cor$-vector spaces on $M$.  We denote by $\Der(\cor_M)$ (resp.\
$\Derb(\cor_M)$) the derived category (resp.\ bounded derived category) of
$\Mod(\cor_M)$.

We recall the definition of the microsupport (or singular support) $\SSi(F)$ of
$F\in \Derb(\cor_M)$, introduced by M.~Kashiwara and P.~Schapira in~\cite{KS82}
and~\cite{KS85}.

\begin{definition}{\rm (see~\cite[Def.~5.1.2]{KS90})}
Let $F\in \Derb(\cor_M)$ and let $p\in T^*M$. 
We say that $p\notin\SSi(F)$ if there exists an open neighborhood $U$ of $p$
such that, for any $x_0\in M$ and any real $C^1$-function $\phi$ on $M$
satisfying $d\phi(x_0)\in U$ and $\phi(x_0)=0$, we have
$(\rsect_{\{x;\phi(x)\geq0\}} (F))_{x_0}\simeq0$.

We set $\dot\SSi(F) = \SSi(F) \cap \dT^*M$.
\end{definition}
In other words, $p\notin\SSi(F)$ if the sheaf $F$ has no cohomology 
supported by ``half-spaces'' whose conormals are contained in a 
neighborhood of $p$. The following properties are easy consequences of
the definition:
\begin{itemize}\renewcommand{\labelitemi}{-}
\item
$\SSi(F)$ is closed and $\R^+$-conic,
\item
$\SSi(F)\cap T^*_MM =\pi_M(\SSi(F))=\supp(F)$,
\item
the triangular inequality: if $F_1\to F_2\to F_3\to[+1]$ is a
distinguished triangle in  $\Derb(\cor_M)$, then 
$\SSi(F_2)\subset\SSi(F_1)\cup\SSi(F_3)$.
\end{itemize}

\begin{example}\label{ex:microsupport}
(i) Let $F\in \Derb(\cor_M)$. Then $\SSi(F) = \emptyset$ if and only if
$F\simeq 0$ and $\dot\SSi(F) = \emptyset$ if and only if the cohomology sheaves
$H^i(F)$ are local systems, for all $i\in \Z$.

\smallskip\noindent
(ii) If $N$ is a smooth closed submanifold of $M$ and $F=\cor_N$, then 
$\SSi(F)=T^*_NM$.

\smallskip\noindent
(iii) Let $\phi$ be $C^1$-function with $d\phi(x)\not=0$ when $\phi(x)=0$.
Let $U=\{x\in M;\phi(x)>0\}$ and let $Z=\{x\in M;\phi(x)\geq0\}$. 
Then 
\begin{align*}
\SSi(\cor_U) & =U\times_MT^*_MM
\cup \{(x;\lambda \, d\phi(x)); \; \phi(x)=0, \, \lambda\leq0\},\\
\SSi(\cor_Z) & =Z\times_MT^*_MM
\cup \{(x;\lambda \, d\phi(x));\; \phi(x)=0, \,  \lambda\geq0\}.
\end{align*}
(iv) Let $\lambda$ be a closed convex cone with vertex at $0$ in $E=\R^n$.
Then $\SSi(\cor_\lambda) \cap T^*_0E = \lambda^\circ$, the polar cone of
$\lambda$, that is,
\begin{equation}\label{eq:def_polar_cone}
\lambda^\circ = \{\xi\in E^*; \; \langle v,\xi \rangle \geq 0
 \text{ for all $v\in E\}$.}
\end{equation}
\end{example}

\subsubsection*{Functorial operations}
Let $M$ and $N$ be two manifolds. We denote by $q_i$ ($i=1,2$) the $i$-th
projection defined on $M\times N$ and by $p_i$ ($i=1,2$) the $i$-th projection
defined on $T^*(M\times N)\simeq T^*M\times T^*N$.

\begin{definition}
Let $f\cl M\to N$ be a morphism of manifolds and let $\Lambda\subset T^*N$
be a closed $\R^+$-conic subset. We say that $f$ is non-characteristic for
$\Lambda$ if $\opb{f_\pi}(\Lambda)\cap T^*_MN\subset M\times_NT^*_NN$.
\end{definition}
A morphism $f\cl M\to N$ is non-characteristic for a closed $\R^+$-conic subset
$\Lambda$ of $T^*N$ if and only if $f_d\cl M\times_NT^*N\to T^*M$ is proper on
$\opb{f_\pi}(\Lambda)$. In this case $f_d\opb{f_\pi}(\Lambda)$ is closed and
$\R^+$-conic in $T^*M$.

We denote by $\omega_M$ the dualizing complex on $M$.  Recall that $\omega_M$
is isomorphic to the orientation sheaf shifted by the dimension. We also use
the notation $\omega_{M/N}$ for the relative dualizing complex
$\omega_M\tens\opb{f}\omega_N^{\tens-1}$.  We have the duality functors
\begin{equation}\label{eq:dualfct}
\DD_M(\scbul)=\rhom(\scbul,\omega_M), \qquad
\DD'_M(\scbul)=\rhom(\scbul,\cor_M).
\end{equation}

\begin{theorem}[See {\cite[\S 5.4]{KS90}}]\label{th:opboim}
Let $f\cl M\to N$ be a morphism of manifolds, $F\in\Derb(\cor_M)$ and
$G\in\Derb(\cor_N)$.  Let $q_1\cl M\times N \to M$ and $q_2\cl M\times N \to N$
be the projections.
\begin{itemize}
\item [(i)] We have
\begin{gather*}
\SSi(\opb{q_1}F \ltens \opb{q_2}G) \subset \SSi(F)\times\SSi(G),  \\
\SSi(\rhom(\opb{q_1}F,\opb{q_2}G))  \subset \SSi(F)^a\times\SSi(G).
\end{gather*}
\item [(ii)] We assume that $f$ is proper on $\supp(F)$. Then
  $\SSi(\reim{f}F)\subset f_\pi\opb{f_d}\SSi(F)$, with equality if $f$ is a
  closed embedding.
\item [(iii)] We assume that $f$ is non-characteristic with respect to
  $\SSi(G)$. Then the natural morphism
  $\opb{f}G \tens \omega_{M/N}\to\epb{f}(G)$ is an isomorphism. Moreover
  $\SSi(\opb{f}G) \cup \SSi(\epb{f}G) \subset f_d\opb{f_\pi}\SSi(G)$.
\item [(iv)] We assume that $f$ is a submersion. Then
  $\SSi(F)\subset M\times_NT^*N$ if and only if, for any $j\in\Z$, the sheaves
  $H^j(F)$ are locally constant on the fibers of $f$.
\end{itemize}
\end{theorem}

\begin{corollary}\label{cor:opboim}
Let $F,G\in\Derb(\cor_M)$. 
\begin{itemize}
\item [(i)] We assume that $\SSi(F)\cap\SSi(G)^a\subset T^*_MM$. Then
  $\SSi(F\ltens G)\subset \SSi(F)+\SSi(G)$.
\item [(ii)] We assume that $\SSi(F)\cap\SSi(G)\subset T^*_MM$. Then \\
  $\SSi(\rhom(F,G))\subset \SSi(F)^a+\SSi(G)$.
\end{itemize}
\end{corollary}

\begin{corollary}\label{cor:opbeqv}
Let $I$ be a contractible manifold and let $p\cl M\times I\to M$ be the
projection.  If $F\in\Derb(\cor_{M\times I})$ satisfies
$\SSi(F)\subset T^*M\times T^*_II$, then $F\simeq\opb{p}\roim{p}F$.
\end{corollary}

The next result follows immediately from Theorem~\ref{th:opboim}~(ii) and
Example~\ref{ex:microsupport}~(i). It is a particular case of the microlocal
Morse lemma (see~\cite[Cor.~5.4.19]{KS90}), the classical theory corresponding
to the constant sheaf $F=\cor_M$.
\begin{corollary}\label{cor:Morse}
Let $F\in\Derb(\cor_M)$, let $\phi\cl M\to\R$ be a function of class $C^1$ and
assume that $\phi$ is proper on $\supp(F)$.  Let $a<b$ in $\R$ and assume that
$d\phi(x)\notin\SSi(F)$ for $a\leq \phi(x)<b$. Then the natural morphisms
$\rsect(\opb{\phi}(]-\infty,b[);F) \to \rsect(\opb{\phi}(]-\infty,a[);F)$
and
$\rsect_{\opb{\phi}([b,+\infty[)}(M;F) \to \rsect_{\opb{\phi}([a,+\infty[)}(M;F)$
are isomorphisms.
\end{corollary}

\section{Simple sheaves on $\R$}
\label{sec:simplesheaves}

Let $\Lambda\subset\dT^*M$ be a locally closed conic Lagrangian submanifold and
let $p\in\Lambda$.  Simple sheaves along $\Lambda$ at $p$ are defined
in~\cite[Def.~7.5.4]{KS90}. Here we only recall a characterization and some
properties of simple sheaves.  For $p\in T^*M$ we denote by $\Derb(\cor_M;p)$
the quotient of $\Derb(\cor_M)$ by the full triangulated subcategory formed by
the $F$ such that $p\not\in \SSi(F)$.

When $\Lambda$ is the conormal bundle to a submanifold $N\subset M$, that is,
when the projection $\pi_M|_\Lambda \cl \Lambda \to M$ has constant rank, then
an object $F\in\Derb(\cor_M)$ is simple along $\Lambda$ at $p$ if
$F\simeq\cor_N\,[d]$ in $\Derb(\cor_M;p)$ for some shift $d\in\Z$.
This means that there exist distinguished triangles
$F' \to F \to L_1 \to[+1]$ and $F' \to \cor_N \to L_2 \to[+1]$
where $p\not\in \SSi(L_i)$, $i=1,2$.

If $\SSi(F)$ is contained in $\Lambda$ on a neighborhood of $\Lambda$,
$\Lambda$ is connected and $F$ is simple at some point of $\Lambda$,
then $F$ is simple at every point of $\Lambda$.

Now we will describe the structure of the simple sheaves on $\R$ with
microsupport contained in the positive direction. We let $(s;\sigma)$ be the
coordinates on $T^*\R$ and we let $T^*_{\sigma>0}\R$ be the subset of $T^*\R$
defined by $\sigma>0$. We let $I= {}]a,b[$ be an interval ($a$ and $b$ may be
$\pm\infty$).  We recall that $\cor$ is a field.

\begin{lemma}\label{lem:deux_decompositions}
Let $\alpha, \beta \in I$ with $\alpha < \beta$.  Let $F,G,H,L \in
\Derb(\cor_I)$. We assume that $\dot\SSi(L) = \emptyset$ and that we have a
distinguished triangle
\begin{align}
\label{eq:td_deux_dec1}
&  F\oplus G \to[u] H \oplus \cor_{[\alpha,\beta[} \to[v] L \to[+1] , \\
\label{eq:td_deux_dec2}
\text{or}\quad & H \oplus \cor_{[\alpha,\beta[} \to F\oplus G \to L \to[+1] .
\end{align}
Then we have a decomposition $F \simeq H_1 \oplus \cor_{[\alpha,\beta[}$ or
$G \simeq H_1 \oplus \cor_{[\alpha,\beta[}$ for some $H_1 \in \Derb(\cor_I)$.
\end{lemma}
\begin{proof}
We give the proof when we have the distinguished
triangle~\eqref{eq:td_deux_dec1}. The case~\eqref{eq:td_deux_dec2} is
similar.
Let $i \cl \cor_{[\alpha,\beta[} \to H \oplus \cor_{[\alpha,\beta[}$ and
$p \cl H \oplus \cor_{[\alpha,\beta[} \to \cor_{[\alpha,\beta[}$ be the natural
morphisms. Let $p_F , p_G \cl F\oplus G \to F\oplus G$ be the
projections to the factors $F$ and $G$ respectively.

Since $L$ has constant cohomology sheaves, we have
$\Hom(\cor_{[\alpha,\beta[} , L) =0$. Hence $v\circ i =0$ and $i$ factorizes
through a morphism $j\cl \cor_{[\alpha,\beta[} \to F\oplus G$. We set
$u_F = p \circ u \circ p_F \circ j$ and $u_G = p \circ u \circ p_G \circ j$.
Then $u_F + u_G = \id_{\cor_{[\alpha,\beta[}}$. Since 
$\Hom(\cor_{[\alpha,\beta[} ,\cor_{[\alpha,\beta[} ) = \cor $,
we deduce that a multiple of $u_F$ or $u_G$ must be $\id_{\cor_{[\alpha,\beta[}}$.
It follows that $\cor_{[\alpha,\beta[}$ is a direct summand of $F$ or $G$.
\end{proof}

\begin{lemma}\label{lem:simplesheaflocal}
We assume that $0\in I$ and
we set $\Lambda = T^*_0\R \cap T^*_{\sigma>0}\R$.
Let $F\in \Derb(\cor_I)$ be such that $\dot\SSi(F) = \Lambda$
and $F$ is simple along $\Lambda$.
Then there exists $M\in \Derb(\cor)$ and $d\in\Z$ such that
$F\simeq \cor_{]a,0[}[d] \oplus M_I$ or $F\simeq \cor_{[0,b[}[d] \oplus M_I$.
\end{lemma}
\begin{proof}
(i) Let $p=(0,1)\in T^*\R$. By definition we have $F\simeq\cor_0\,[\delta]$ in
$\Derb(\cor_M;p)$ for some $\delta\in\Z$.  The functor
$(\rsect_{[0,+\infty[}(\cdot))_0$ vanishes on the $F$ with $p\not\in\SSi(F)$, by
definition of the microsupport. Hence it is well-defined in $\Derb(\cor_M;p)$
and we find $(\rsect_{[0,+\infty[}F)_0 \simeq \cor[\delta]$.
The image of $1\in \cor$ by this isomorphism gives a morphism
$v\cl \cor_{[0,\varepsilon[}[-\delta] \to F|_J$ defined on some neighborhood
$J = {}]{-\varepsilon},\varepsilon[$ of $0$.  Then, defining $L$ on $J$ and
$u\cl L \to \cor_{[0,\varepsilon[}[-\delta]$ by the distinguished triangle
$L \to[u] \cor_{[0,\varepsilon[}[-\delta] \to[v] F|_J \to[+1]$, we have
$\SSi(L) \subset T^*_JJ$.

\medskip\noindent
(ii) If $u=0$, we obtain $F|_J \simeq L \oplus \cor_{[0,\varepsilon[}[-\delta]$.
If $u\not=0$, then we can decompose $L \simeq \cor_J[-\delta] \oplus L'$ (by
splitting $u_x$ for some $x\in \mo{]}0,\varepsilon[$) so that $u$ is induced by
the projection $L \to \cor_J[-\delta]$ composed with
$\cor_J \to \cor_{[0,\varepsilon[}$.
We deduce $F|_J \simeq L' \oplus \cor_{]-\varepsilon,0[}[1-\delta]$.  Since $F$ is
constant outside $0$ we deduce the lemma.
\end{proof}

We let $A=\{s_1,\dots,s_k\}$ be a finite subset of $I$ and we set
$\Lambda = (\bigsqcup_{i=1}^k T^*_{s_i}I ) \cap T^*_{\sigma>0}\R$.

\begin{proposition}\label{prop:simplesheafR}
{\rm(i)} Let $F\in \Derb(\cor_I)$ be such that $\dot\SSi(F) = \Lambda$
and $F$ is simple along $\Lambda$.
Then, up to reordering the indices of the $s_i$'s, there exists
an isomorphism
\begin{equation}\label{eq:simplesheafR0}
F \simeq L \oplus \bigoplus_{i=1}^l \cor_{[s_{2i-1},s_{2i}[} [d_i]
\oplus \bigoplus_{i=2l+1}^m \cor_{[a,s_i[} [d_i]
\oplus \bigoplus_{i=m+1}^k \cor_{[s_i,b[}  [d_i] ,
\end{equation}
for some integers $d_i$ and some $L\in\Derb(\cor_\R)$ with constant
cohomology sheaves.

\smallskip\noindent
{\rm(ii)} Using the notations of~\eqref{eq:simplesheafR0} we set
$S^\infty(F) = \{s_i$; $i=2l+1,\dots,k\}$. Then $S^\infty(F)$ only depends on $F$.
Moreover, for any distinguished triangle $F \to F' \to L' \to[+1]$ where
$\dot\SSi(L') =\emptyset$, we have $F'$ simple along $\Lambda$ and
$S^\infty(F') = S^\infty(F)$.
\end{proposition}
\begin{proof}
(i-a) Let us first assume that $F$ is concentrated in degree $0$.  Let us
proceed by induction on $k=|A|$. The case $k=1$ is given by
Lemma~\ref{lem:simplesheaflocal}. If $k>1$, let $s_1<s_2$ be the first two
elements of $A$. By Lemma~\ref{lem:simplesheaflocal} we have either
$F|_{]a,s_2[} \simeq L \oplus \cor_{]a,s_1[}$ or
$F|_{]a,s_2[} \simeq L \oplus \cor_{[s_1,s_2[}$, for some constant sheaf $L$
on $]a,s_2[$.

\medskip\noindent
(i-b) If $F|_{]a,s_2[} \simeq L \oplus \cor_{]a,s_1[}$, then this decomposition
immediately extends to $F \simeq \cor_{]a,s_1[} \oplus G$, where $G$ satisfies
the same hypothesis as $F$ with $A$ replaced by $A\setminus \{s_1\}$.  Then the
induction hypothesis gives the result.

\medskip\noindent
(i-c) Now we assume that $F|_{]a,s_2[} \simeq L \oplus \cor_{[s_1,s_2[}$ and we let
$u\cl \cor_{[s_1,s_2[} \to F$ be the morphism induced by this decomposition.
Let $s\in (A\setminus \{s_1\}) \cup \{b\}$ be maximal such that there exists a
monomorphism $v\cl \cor_{[s_1,s[} \to F$ extending $u$.
We define $G$ by the exact sequence
\begin{equation}\label{eq:simplesheafR1}
  0 \to \cor_{[s_1,s[} \to F  \to G \to 0 .
\end{equation}
Using Lemma~\ref{lem:simplesheaflocal} around $s$, we see that
$G$ satisfies the same hypothesis as $F$, with $A$ replaced by
$A\setminus \{s_1,s\}$ (if $s$ is one of the $s_i$'s) or $A\setminus \{s_1\}$
(if $s=b$). By the induction hypothesis $G$ is a sum of sheaves of the type
$\cor_{[s',s''[}$ with $s',s''\in A\cup \{a,b\}$ and $s''\not=s_1$.  We remark
that $\Ext^1(\cor_{[x,y[},\cor_{[z,w[}) \simeq 0$ if $x\not= w$.  Hence the
exact sequence~\ref{eq:simplesheafR1} splits and we obtain the result.

\medskip\noindent
(i-d) For a general $F$ we deduce from Lemma~\ref{lem:simplesheaflocal} that
each $H^iF$ satisfies the same hypothesis as $F$, with $A$ replaced by some
subset of $A$. Hence we know the structure of $H^iF$ by~(i-a)-(i-c).
We deduce easily that $\Ext^p(H^iF,H^jF) \simeq 0$ for all 
$i,j$ and $p\geq 2$. This implies that $F \simeq \bigoplus_i H^iF[-i]$
and we obtain the result.

\medskip\noindent
(ii) Lemma~\ref{lem:deux_decompositions} implies that $F$ and $F'$ have the
same direct summands $\cor_{[s_{2i-1},s_{2i}[} [d_i]$, $i=1,\ldots,l$.
This is equivalent to~(ii).
\end{proof}

\section{The involutivity theorem}

The main tool in our proof of the Gromov-Eliashberg theorem is the involutivity
theorem of~\cite{KS90}.  This is a deep result originally inspired by the
similar theorem for the characteristic variety of a system of linear PDE's
(see {\em loc. cit.} for historical comments on this point).

We first recall a general definition of involutivity given in~\cite{KS90}.
Let $X$ be a manifold and let $x\in X$, $S\subset X$ be a point and a subset of
$X$.  We denote by $C_x(S) \subset T_xX$ the tangent cone of $S$ at $x$.  In
case $X$ is a vector space this is the set of $v \in X \simeq T_xX$ which can
be written $v = \lim_{n\to\infty} c_n(x_n-x)$, for some sequences
$\{c_n\}_{n\in \N}$ and $\{x_n\}_{n\in \N}$ with $c_n\in \R^+$, $x_n\in S$
satisfying $x = \lim_{n\to\infty} x_n$.
For two subsets $S_1,S_2 \subset X$ we also have $C_x(S_1,S_2) \subset T_xX$
(see~\cite{KS90}). In case $X$ is a vector space this is the set of $v$ which
can be written $v = \lim_{n\to\infty} c_n(x^1_n-x^2_n)$, for some sequences
$\{c_n\}_{n\in \N}$ and $\{x^i_n\}_{n\in \N}$ with $c_n\in \R^+$, $x^i_n\in S_i$,
$i=1,2$, satisfying $x = \lim_{n\to\infty} x^i_n$.

If $(E,\omega)$ is a symplectic vector space and $A\subset E$ we set
$A^{\perp\omega} = \{v\in E$; $\omega(v,w) = 0$, for all $w\in A\}$.

\begin{definition}[Def.~6.5.1 of~\cite{KS90}]\label{def:invol}
Let $(X,\omega)$ be a symplectic manifold and let $S$ be a locally closed subset
of $X$. For a given $p\in S$ we say that $S$ is coisotropic (or involutive) at
$p$ if $(C_p(S,S))^{\perp\omega_p} \subset C_p(S)$.
\end{definition}

\begin{theorem}[Thm.~6.5.4 of~\cite{KS90}]\label{thm:invol}
Let $M$ be a manifold and $F\in \Derb(\cor_M)$.
Then $\SSi(F)$ is coisotropic.
\end{theorem}

We will use the following results when we apply Theorem~\ref{thm:invol}
in Section~\ref{sec:GET}.
\begin{lemma}\label{lem:inclu_cois}
Let $X$ be a symplectic manifold and let $S \subset S'$ be locally closed
subsets of $X$. Let $p\in S$. We assume that $S$ is coisotropic at $p$.
Then $S'$ is also coisotropic at $p$.
\end{lemma}
\begin{proof}
We have the inclusions\\
$(C_p(S',S'))^{\perp\omega_p} \subset (C_p(S,S))^{\perp\omega_p} \subset
C_p(S) \subset C_p(S')$.
\end{proof}

\begin{proposition}\label{prop:invol_invol}
Let $M$ be a manifold and $S\subset T^*M$ a locally closed subset.
We recall the map $\rho_M \cl T^*M\times\dT^*\R \to T^*M$
defined in~\eqref{eq:def_rho}.
Let $p\in S$ and $q\in \opb{\rho_M}(p)$.
Then $S$ is coisotropic at $p$ if and only if
$\opb{\rho_M}(S)$ is coisotropic at $q$.
\end{proposition}
\begin{proof}
We use coordinates $(x,s;\xi,\sigma)$ on $T^*(M\times\R)$ and
corresponding coordinates $(X,S;\Xi,\Sigma)$ on $T_qT^*(M\times\R)$.
We set $S' = \opb{\rho_M}(S)$ and we write $q=(x_0,s_0;\xi_0,\sigma_0)$.
We have $d\rho_{M,q}(X,S;\Xi,\Sigma)
 = (X; \frac{1}{\sigma_0} \Xi - \frac{\xi_0}{\sigma_0^2} \Sigma)$.
Since $S'$ is conic, we may assume $\sigma_0=1$.
Using the symplectic transformations $(x;\xi) \mapsto (x;\xi-\xi_0)$
on $T^*M$ and $(x,s;\xi,\sigma)\mapsto
(x,s+\langle \xi_0,x\rangle; \xi-\sigma \xi_0,\sigma)$
on $T^*(M\times \R)$ we may also assume $\xi_0=0$.
Then we have $d\rho_q (X,S;\Xi,\Sigma) = (X;\Xi)$ and we deduce
$C_q(S') = C_p(S) \times T_{(s_0;1)}T^*\R$
and $C_q(S',S') = C_p(S,S) \times T_{(s_0;1)}T^*\R$.
Now the result follows easily.
\end{proof}

\section{Bounds for microsupports}

For a real vector space $V$ and $A\subset V$ we denote by $\Conv(A)$ the convex
hull of $A$. Let $X$ be a manifold.  For $\Lambda \subset T^*X$ we define
$\Conv(\Lambda) \subset T^*X$ by
$\Conv(\Lambda) \cap T^*_xX = \Conv(\Lambda\cap T^*_xX)$, for all $x\in X$.  In
general it can happen that $\Lambda$ is closed but not $\Conv(\Lambda)$.  We
leave the following result to the reader.
\begin{lemma}\label{lem:Conv_ferme}
Let $\Lambda \subset T^*X$ be a closed conic subset.
We assume that $\Conv(\Lambda\cap T^*_xX)$ is proper for all $x\in X$,
that is, $\Conv(\Lambda\cap T^*_xX)$ contains no line.
Then $\Conv(\Lambda)$ is closed.
\end{lemma}

We will use the following cut-off result (other similar and more precise
results are recalled in Section~\ref{sec:cutoff}).
Let $V$ be a vector space and let $\gamma \subset V$ be a closed convex
cone. Let $V_\gamma$ be the space $V$ endowed with the topology whose open sets
are the usual open subsets $\Omega$ of $V$ such that $\Omega+\gamma=\Omega$.
Let $\Mod(\cor_{V_\gamma})$ be the category of sheaves of $\cor$-vector spaces
on $V_\gamma$ and $\Derb(\cor_{V_\gamma})$ its bounded derived category.  The
identity map induces a continuous map $\phi_\gamma \cl V \to V_\gamma$.

\begin{proposition}[see \S3.5 and Prop.~5.2.3 of~\cite{KS90}]
\label{prop:cut-off-bis}
The inverse image $\opb{\phi_\gamma} \cl \Derb(\cor_{V_\gamma}) \to \Derb(\cor_V)$
induces an equivalence between $\Derb(\cor_{V_\gamma})$ and the full
subcategory of $\Derb(\cor_V)$ consisting of the $F$ such that
$\SSi(F) \subset V \times \gamma^{\circ a}$.
\end{proposition}

The following consequence was pointed out to the author by Pierre Schapira.
\begin{proposition}\label{prop:boundSStronc}
Let $F\in \Derb(\cor_X)$ be such that $\Conv(\SSi(F)\cap T^*_xX)$ is proper
for all $x\in X$.
Then, for all $n\in \Z$, the microsupports of $\tau_{\leq n} F$,
$\tau_{\geq n} F$ and $H^n(F)$ are contained in $\Conv(\SSi(F))$.
\end{proposition}
\begin{proof}
(i) We set $\Lambda = \SSi(F)$.
Let $G\in \Derb(\cor_X)$ be one of $\tau_{\leq n} F$, $\tau_{\geq n} F$ or
$H^n(F)$. We prove that $\SSi(G) \cap T^*_{x_0}X$ is contained in
$\Conv(\Lambda\cap T^*_{x_0}X)$ for any $x_0 \in X$.
Since this is a local problem around $x_0$ we may assume that $X$ is an open
subset of a vector space, say $V$, and that $T^*X \simeq X \times V^*$.
We also assume that $x_0 = 0$.

\medskip\noindent
(ii) We let $\gamma \subset V$ be a closed convex proper cone of $V$ such that
$\gamma^{\circ a}$ is a neighborhood of
$\Conv(\Lambda\cap T^*_0X) \setminus \{0\}$.
We can find a neighborhood of $0$, say $U$, such that
$\Lambda \cap T^*U \subset U \times \gamma^{\circ a}$.
We choose coordinates $(x_1,\ldots,x_n)$ on $V$ such that
$(0,\ldots,0, -1) \in \gamma$. We choose an open convex cone $\gamma'$
such that $\ol{\gamma'}$ is proper and
$\gamma \setminus \{0\} \subset \gamma'$.
For $\varepsilon>0$ we define
$\gamma'_\varepsilon 
= ( (0,\ldots,0,\varepsilon) + \gamma') \cap \{x_n \geq -\varepsilon\}$.
Then $\gamma'_\varepsilon$ is a neighborhood of $0$ in $V$ and,
for $\varepsilon$ small enough, we have $\ol{\gamma'_\varepsilon} \subset U$.
Moreover $\SSi(\cor_{\gamma'_\varepsilon})
\subset \ol{\gamma'_\varepsilon}\times \gamma^{\circ a}$.

\medskip\noindent
(iii) We set $F' = F \tens \cor_{\gamma'_\varepsilon}$. Then $F'$ is isomorphic
to $F$ on $\Int \gamma'_\varepsilon$ and
$\SSi(F') \subset V \times \gamma^{\circ a}$ by
Corollary~\ref{cor:opboim}. Hence, by Proposition~\ref{prop:cut-off-bis} there
exists $F_1\in \Derb(\cor_{V_\gamma})$ such that
$F' \simeq \opb{\phi_\gamma}(F_1)$.
Then $G|_{\Int \gamma'_\varepsilon}
\simeq \opb{\phi_\gamma}(G_1)|_{\Int \gamma'_\varepsilon}$, where $G_1$ is one of
$\tau_{\leq n} F_1$, $\tau_{\geq n} F_1$ or $H^n(F_1)$. 
By Proposition~\ref{prop:cut-off-bis} again we deduce
$\SSi(G|_{\Int \gamma'_\varepsilon}) \subset V \times \gamma^{\circ a}$.
Since $\gamma^{\circ a}$ is an arbitrary small neighborhood of
$\Conv(\Lambda\cap T^*_0X) \setminus \{0\}$, we obtain that
$\SSi(G) \cap T^*_0X$ is contained in $\Conv(\Lambda\cap T^*_0X)$, as required.
\end{proof}

\begin{proposition}\label{prop:boundSSlim}
Let $X$ be a manifold and $a \leq b \in \Z$.
Let $\{K_n\}_{n\in\N}$ be a family of objects of $\Derb(\cor_X)$ such that
$H^iK_n = 0$ for $i\not\in [a,b]$ and for all $n\in\N$.
We define $K\in \Derb(\cor_X)$ by the distinguished triangle
$$
\bigoplus_{n\in\N} K_n \to \prod_{n\in\N} K_n \to K \to[+1].
$$
Then we have $\SSi(K) \subset \bigcap_{k\in \N} \ol{\bigcup_{n\geq k} \SSi(K_n)}$.
\end{proposition}
\begin{proof}
For any $k\in \N$ we also have a distinguished triangle
$\bigoplus_{n\geq k} K_n \to \prod_{n\geq k} K_n \to K \to[+1]$.
We can check, similarly as in~\cite[Exe. V.7]{KS90}, that
$\SSi(\bigoplus_{n\geq k} K_n) \subset \ol{\bigcup_{n\geq k} \SSi(K_n)}$
and $\SSi(\prod_{n\geq k} K_n) \subset \ol{\bigcup_{n\geq k} \SSi(K_n)}$.
We conclude by the triangular inequality for the microsupport.
\end{proof}

\section{Approximation of symplectic maps}

Let $(E,\omega)$ be a symplectic vector space which we identify with $\R^{2n}$.
We endow $E$ with the Euclidean norm of $\R^{2n}$.

\begin{lemma}\label{lem:appr_symplC1}
Let $R>r$ and $\varepsilon$ be positive numbers.  Let $\varphi \cl B_R^E\to E$
be a symplectic map of class $C^1$. Then there exists $R'>r$ and a symplectic
map $\psi \cl B_{R'}^E\to E$ which is of class $\Cinf$ such that
$\| \varphi - \psi \|_{B_r^E} \leq \varepsilon$.
\end{lemma}
\begin{proof}
We set $r_1=(R+r)/2$ and we choose a (non symplectic) map
$\varphi' \cl B_R^E\to E$ of class $\Cinf$ such that
$\| \varphi - \varphi' \|^1_{B_{r_1}^E} \leq \varepsilon$
(we use the norm~\eqref{eq:norm1}).
We set $\omega' = \varphi'^*(\omega)$.
We have $\omega - \omega' = (\varphi - \varphi')^*\omega$.
Hence, if we consider $\omega$ and $\omega'$ as maps from $E$ to $\wedge^2E$
and we endow $\wedge^2E$ with the Euclidean structure induced by $E$, we have
$\| \omega - \omega' \|_{B_{r_1}^E} \leq C\varepsilon$, where the constant
$C$ only depends on $n$.

We set $r_2 = (r_1+r)/2$.  By Moser's argument for the Darboux theorem we can
find a flow $\Phi \cl B_{r_1}^E \times [0,1] \to E$ such that
$\Phi_t(B_{r_2}^E) \subset B_{r_1}^E$ for all $t\in [0,1]$ and
$\omega|_{B_{r_2}^E} = \Phi_1^*(\omega')|_{B_{r_2}^E}$.
The flow $\Phi$ is the flow of a vector field $X_t$ which satisfies
$\iota_{X_t}(\omega_t) = -\sigma$ over $B_{r_1}^E$, where
$\omega_t = t \omega' - (1-t)\omega$ and $d\sigma = \omega' - \omega$.
We can assume that $\sigma$ satisfies the bound
$\| \sigma \|_{B_{r_1}^E} \leq C'\| \omega' - \omega \|_{B_{r_1}^E}$
for some $C'>0$ only depending on $r_1$.
Hence $X_t$ satisfies $\| X_t \|_{B_{r_1}^E} \leq C''\varepsilon$, for some
constant $C''>0$ and all $t\in [0,1]$.

We may assume from the beginning that $C''\varepsilon < r_1-r_2$.
Hence $\Phi_1(B_{r_2}^E) \subset B_{r_1}^E$ and we have
$\| \Phi_1 - \id \|_{B_{r_2}^E} \leq C''\varepsilon$.
The map $\psi = \varphi' \circ \Phi_1 \cl B_{r_2}^E \to E$ is a symplectic map
such that $\| \varphi - \psi \|_{B_r^E} \leq (1+C'')\varepsilon$, which
gives the lemma (up to replacing $\varepsilon$ by $\varepsilon/(1+C'')$).
\end{proof}

\begin{proposition}\label{prop:appr_Hamisot}
Let $R>r$ and $\varepsilon$ be positive numbers.  Let $\varphi \cl B_R^E\to E$
be a symplectic map of class $C^1$.
Then there exists a Hamiltonian isotopy $\Phi \cl E\times \R \to E$ of class
$\Cinf$ and a compact subset $C \subset E$ such that
$\| \varphi - \Phi_1 \|_{B_r^E} \leq \varepsilon$ and
$\Phi_t |_{E\setminus C} = \id _{E\setminus C}$ for all $t\in\R$.
\end{proposition}
\begin{proof}
(i) By Lemma~\ref{lem:appr_symplC1} we may assume that $\varphi$ is of class
$\Cinf$. Composing with a translation in $E$ and a symplectic linear map, we
may also assume that $\varphi(0) = 0$ and $d\varphi_0 = \id_E$.  We first show
that there exists a symplectic map $\varphi' \cl B_R^E\to E$ such that
$\| \varphi - \varphi' \|_{B_R^E} \leq \varepsilon$ and $\varphi'= \id_E$ near
$0$.

\medskip\noindent
(ii) We choose an isomorphism $E\simeq V\times V^*$ and then
$E^2 \simeq T^*W \simeq W\times W^*$, where $W=V\times V^*$ is the product
of first and last factors in $E^2= V\times V^* \times V\times V^*$ (of course
$W\simeq E$). We have a natural isomorphism $s\cl W \isoto W^*$ given by
switching the factors $V$ and $V^*$.

For a function $f\cl W\to\R$ we write $\Lambda^f = \{(x,df_x)$; $x\in W \}$.
We can find two balls centered around $0$, $B_{r_1}^W$ in $W$ and
$B_{r_2}^{W^*}$ in $W^*$ and a function $f$ defined on $B_{r_1}^W$
such that $\Lambda_\varphi \cap (B_{r_1}^W \times B_{r_2}^{W^*})
= \Lambda^f \cap (B_{r_1}^W \times B_{r_2}^{W^*})$.
We have $df_0 = 0$ and we assume $f(0)=0$.
We choose $\eta>0$ and $f'$ defined on $B^W_{4\eta}$ such that
\begin{itemize}
\item $\|f-f'\|^1_{B^W_{4\eta}} < \varepsilon$,
\item $f'(x,\xi) = \langle x,\xi \rangle$, hence $df' = s$, on $B^E_\eta$,
\item $f'=f$ outside of $B^E_{2\eta}$.
\end{itemize}
Then we can define a symplectic map $\varphi' \cl B_R^E\to E$ such that
$\Lambda_{\varphi'} = \Lambda_\varphi$ outside of $B^W_{4\eta}\times B_{r_2}^{W^*}$ and
$\Lambda_{\varphi '} \cap (B^W_{4\eta}\times B_{r_2}^{W^*})
= \Lambda^{f'} \cap (B^W_{4\eta}\times B_{r_2}^{W^*})$.
We have $\| \varphi - \varphi'  \|_{B_R^E} \leq \varepsilon$ 
and $\varphi'|_{B^E_\eta} = \id_E$,  as claimed in~(i).

\medskip\noindent
(iii) We will prove that $\varphi'|_{B^E_r}$ is the time $1$ of some
Hamiltonian isotopy.
We define $U \subset E \times ]0,+\infty[$ by
$U = \{(x,t)$; $\| x \| < R/t \}$ and $\psi \cl U \to E$
by $\psi(x,t) = t^{-1}\varphi'(tx)$.
We let $U' = \{(\psi(x,t),t)$; $(x,t) \in U\}$ be the image of $U$ by
$\psi \times \id_\R$.
Then $U'$ is contractible and we can find $h\cl U' \to \R$ such that
$\psi$ is the Hamiltonian flow of $h$.

We define $U_0\subset U$ by $U_0 = \{(x,t)$; $t>0$, $\| x \| < \eta/t \}$.
Since $\varphi'|_{B^E_\eta} = \id_E$, we have $\psi(x,t) = x$ for all
$(x,t) \in U_0$. Hence $U_0 \subset U'$. Moreover $h$ is constant on $U_0$.
We can assume $h|_{U_0} =0$ and extend $h$ by $0$ to a $\Cinf$ function
defined on $U'' = \mo{]}-\infty,0] \cup U'$.

We set $Z = \{(\psi(x,t),t)$; $t\in{} ]0,1]$ and $\|x\| \leq r\}$.
We remark that $Z\cap U_0 = (\ol{B_r^E} \times {}]0,1]) \cap U_0$.
Hence $Z$ is relatively compact in $U''$.
We choose a compact subset $C\subset E$ such that $Z \subset C\times [0,1]$ and
a $\Cinf$ function $g\cl E\times \R \to \R$ such that $g=h$ on $Z$ and $g=0$
outside of $C\times [0,2]$. Then the Hamiltonian isotopy $\Phi$ defined by $g$
has compact support contained in $C$ and satisfies $\Phi_1 = \varphi'$ on
$B_r^E$.  This proves the proposition.
\end{proof}

When we apply Proposition~\ref{prop:appr_Hamisot} above we have no control on
the extension of the map outside the ball $B_r^E$.  In the following lemmas we
check that, up to a linear transformation, the ``interesting part'' of the
graph can be moved near the zero section of $T^*E^2$ and the ``extended part''
far away from the zero section (see~\eqref{eq:gen_pos_ineq}
and~\eqref{eq:proche_gen_pos}).

\begin{lemma}\label{lem:gen_pos}
Let $R_0>0$ and let $\varphi \cl B_{R_0}^E \to E$ be a map of class $C^1$ such
that $\varphi(0) = 0$ and $d\varphi_0$ is an isomorphism. We set $V=\R^n$.
Then we can find symplectic linear isomorphisms 
$u\cl T^*V \isoto E$, $v\cl E\to T^*V$ and positive numbers $r_0, A$ such that,
setting $\psi = v\circ \varphi \circ u \cl \opb{u}(B_{R_0}^E) \to T^*V$,
we have
\begin{align}
\label{eq:gen_pos_ineq0}
& \text{$\psi$ is defined on $B_{r_0}^V \times B_{Ar_0}^{V^*}$}, \\
\label{eq:gen_pos_ineq}
&  \Gamma_{\psi} \cap (B_{r}^{V^2} \times B_{Ar_0}^{V^{2*}})
\subset  B_{r}^{V^2} \times B_{Ar}^{V^{2*}} 
\qquad \text{for any $r_0 \geq r > 0$}, \\
& \label{eq:gen_pos_ineq2}
\text{the projection
$\Gamma_{\psi} \cap (B_{r_0}^{V^2} \times B_{Ar_0}^{V^{2*}}) \to B_{r_0}^{V^2}$
 is a diffeomorphism.}
\end{align}
\end{lemma}
\begin{proof}
We first choose a symplectic isomorphism  $u\cl T^*V \isoto E$
arbitrarily and we set $L= d\varphi_0( V \times \{0\} )$.
Since $L$ is of dimension $n$, we can find a Lagrangian subspace
$L' \subset E$ such that $L\cap L' =\{0\}$.
We choose a symplectic diffeomorphism $v\cl E \to T^*V$ such that
$v(L') = V \times \{0\} $.

We set $\psi = v\circ \varphi \circ u$,
$W = T_{(0,0)}\Gamma_\psi \subset T^*V^2$ and \\
$W_1 = \{ ((x_1;0), d\psi_0(x_1;0))$; $x_1\in V \}$,
$W_2 = \{ (\opb{(d\psi_0)}(x_2;0), (x_2;0) )$; $x_2\in V \}$. 
We also set $p_i(x_1,x_2;\xi_1,\xi_2) = x_i$, $i=1,2$.
Then $W = W_1 \oplus W_2$, $p_1(W_1) = V$ and $p_2(W_2) = V$.
Hence $\pi_{V^2} \cl T^*V^2 \to V^2$ maps $W$ onto $V^2$.

In particular $W$ is the graph of a linear map $w\cl V^2 \to V^{*2}$.
Setting $A' = \|w\|$ we have $W \cap (B_{r}^{V^2} \times V^{*2})
\subset B_{r}^{V^2} \times B_{A'r}^{V^{2*}}$ for all $r>0$.
Since $\Gamma_\psi$ is a manifold of class $C^1$, a similar result
holds for $\Gamma_\psi$ near $0$ with some $A>A'$. The lemma follows.
\end{proof}

\begin{lemma}\label{lem:proche_gen_pos}
Let $V$ be a vector space. Let $A, r_0>0$ be given and set
$U= B_{r_0}^V \times B_{Ar_0}^{V^*} \subset T^*V$.
Let $\psi \cl U \to T^*V$ be a map which
satisfies~\eqref{eq:gen_pos_ineq}.  We choose $0 < r < r_0/4$ and
$0 < \varepsilon < Ar/(A+1)$.
Then, for any map $\psi_1 \cl T^*V \to T^*V$ satisfying
\begin{equation}
\label{eq:psi_proche_phi}
\text{$d(\psi(p), \psi_1(p)) < \varepsilon$ for all $p\in U$} ,
\end{equation}
we have
\begin{equation}
\label{eq:proche_gen_pos}
\Gamma_{\psi_1} \cap (B_{r}^{V^2} \times B_{Ar_0/2}^{V^{2*}})
\subset \Gamma_{\psi_1|_U} \cap \Gamma_{\psi}(\varepsilon) \cap
(B_{r}^{V^2} \times B_{2Ar}^{V^{2*}} ) ,
\end{equation}
where $\Gamma_{\psi}(\varepsilon) = \{p\in T^*E^2;$
$d(p,\Gamma_{\psi}) < \varepsilon \}$.
\end{lemma}
\begin{proof}
Let $(x;\xi) \in T^*V$ and $(y_1;\eta_1) = \psi_1(x;\xi)$.  We assume that
$q=(x,y_1;\xi,\eta_1)$ is in the LHS of~\eqref{eq:proche_gen_pos}.

\medskip\noindent
(i) We have in particular $(x;\xi) \in B_{r}^{V} \times B_{Ar_0/2}^{V^{*}}$.
Hence $(x;\xi) \in U$ and $q\in \Gamma_{\psi_1|_U}$.
Moreover~\eqref{eq:psi_proche_phi} gives $q\in \Gamma_{\psi}(\varepsilon)$.

\medskip\noindent
(ii) It remains to prove that $(\xi,\eta_1) \in B_{2Ar}^{V^{2*}}$.
 Let us write $(y;\eta) = \psi(x;\xi)$.
By~\eqref{eq:psi_proche_phi} we have $d(y_1,y) < \varepsilon$ and
$d(\eta_1,\eta) < \varepsilon$.
Hence $(x,y) \in B_{r + \varepsilon}^{V^{2}}$
and  $(\xi,\eta) \in B_{Ar_0/2 + \varepsilon}^{V^{2*}}$.
The hypothesis on $r$ and $\varepsilon$ implies
$$
r+\varepsilon <  r_0, \quad
Ar_0/2 + \varepsilon <  Ar_0, \quad
A(r +\varepsilon) + \varepsilon < 2Ar.
$$
By~\eqref{eq:gen_pos_ineq} we deduce $(\xi,\eta) \in B_{A(r+\varepsilon)}^{V^{2*}}$.
Since $d(\eta_1,\eta) < \varepsilon$, we obtain
$(\xi,\eta_1) \in B_{A(r+\varepsilon) +\varepsilon}^{V^{2*}}$
and this proves the result.
\end{proof}

\section{Degree of a continuous map}

We recall the definition of the degree of a continuous map.  Let $M,N$ be two
oriented manifolds of the same dimension, say $d$.  We assume that $N$ is
connected. We have a morphism $H^d_c(M;\Z_M) \to \Z$ and an isomorphism
$H^d_c(N;\Z_N) \isoto \Z$.
Let $f\cl M\to N$ be a proper continuous map.  Applying $H^d_c(N;\cdot)$ to the
morphism $\Z_N \to \roim{f} \opb{f}\Z_N \simeq \reim{f} \Z_M$ we find
$$
\Z \isofrom H^d_c(N;\Z_N) \to H^d_c(M;\Z_M) \to \Z .
$$ 
The degree of $f$, denoted $\deg f$, is the image of $1$ by this morphism.

\begin{lemma}\label{lem:degree}
Let $M,N$ be two oriented manifolds of dimension $d$.  We assume
that $N$ is connected.

\noindent
{\rm(i)} Let $f\cl M\to N$ be a proper continuous map
and let $V \subset N$ be a connected open subset.
Then $\deg f = \deg(f|_{\opb{f}(V)} \cl \opb{f}(V) \to V)$.

\noindent
{\rm(ii)} Let $I\subset \R$ be an interval.
Let $U \subset M\times I$, $V \subset N\times I$ be open subsets and
let $f\cl U \to V$ be a continuous map which commutes with the
projections $U\to I$ and $V \to I$. We set
$U_t = U\cap (M\times \{t\})$, $V_t = V\cap (N\times \{t\})$ and
$f_t = f|_{U_t} \cl U_t \to V_t$, for all $t\in I$. We assume that $f$
is proper and that $V$ and all $V_t$, $t\in I$, are non empty and
connected.
Then $\deg f = \deg f_t$, for all $t\in I$.
\end{lemma}
\begin{proof}
(i) and (ii) follow respectively from the commutative diagrams
$$
\begin{tikzcd}[row sep=6mm]
\Z \dar{=} & H^d_c(N;\Z_V) \lar[swap]{\sim} \rar \dar 
  & H^d_c(M;\Z_{\opb{f}(V)}) \rar  \dar  & \Z \dar{=}   \\
\Z & H^d_c(N;\Z_N) \lar[swap]{\sim} \rar & H^d_c(M;\Z_M) \rar  & \Z \virgdiag
\end{tikzcd}
$$
$$
\begin{tikzcd}[row sep=6mm]
\Z \dar{=} & H^d_c(V_t;\Z_{V_t}) \lar[swap]{\sim} \rar \dar 
  & H^d_c(U_t;\Z_{U_t}) \rar  \dar  & \Z \dar{=}   \\
\Z  & H^{d+1}_c(V;\Z_V) \lar[swap]{\sim} \rar  & H^{d+1}_c(U;\Z_U) \rar  
 & \Z \pointdiag
\end{tikzcd}
$$
\end{proof}

\begin{proposition}\label{prop:degree}
Let $B_R$ be the open ball of radius $R$ in $\R^d$. Let $U, V \subset \R^d$ be
open subsets and let $f \cl U \to B_R$, $g \cl V \to B_R$ be proper
continuous maps. We assume that there exists $r<R$ such that
$\opb{f}(\ol{B_r}) \subset U\cap V$, and that $d(f(x),g(x)) < r/2$, for all
$x\in U\cap V$.  Then $\deg f = \deg g$.
\end{proposition}
\begin{proof}
(i) We define $h\cl (U\cap V) \times [0,1] \to \R^{d+1}$ by
$h(x,t) = (tf(x) + (1-t)g(x),t)$.  Let us prove that
$\opb{h}( \ol{B_{r/2}}\times [0,1])$ is compact.
Since $\opb{f}(\ol{B_r})$ is compact and contained in $U\cap V$,
it enough to prove that
$\opb{h}( \ol{B_{r/2}}\times [0,1]) \subset \opb{f}(\ol{B_r}) \times [0,1]$.
Let $(x,t) \in (U\cap V) \times [0,1]$ be such that $\| h(x,t) \| \leq r/2$.
Since $h(x,t)$ belongs to the line segment $[f(x),g(x)]$ which is of length
$< r/2$, we deduce $f(x) \in B_r$, as required.

\medskip\noindent
(ii) We define $W =\opb{h}(B_{r/2} \times [0,1])$,
$W_t = W\cap (\R^d\times \{t\})$ for $t\in [0,1]$ and
$h'_t = h|_{W_t} \cl W_t \to B_{r/2}$.
By~(i) $h|_W \cl W \to B_{r/2} \times [0,1]$ is proper.
Hence Lemma~\ref{lem:degree}~(ii) implies that $\deg h'_0 = \deg h'_1$.
We conclude with Lemma~\ref{lem:degree}~(i) which implies
$\deg h'_0 = \deg g$ and $\deg h'_1 = \deg f$.
\end{proof}

\section{Cut-off}
\label{sec:cutoff}
In this section we recall several results of~\cite{KS90} which are called
``(dual) (refined) cut-off lemmas''.  In {\em loc. cit.} the statements of the
refined cut-off lemmas are given around a point.  With stronger hypothesis on
the microsupports they hold on a fixed neighborhood of a given point, which is
needed in the next section. We include the part of the proof which is concerned
with this claim but it is actually the same as in {\em loc. cit.}

Let $V$ be a vector space and let $\gamma \subset V$ be a closed convex cone
(with vertex at $0$). We denote by $\gamma^a = -\gamma$ its opposite cone and
by $\gamma^\circ \subset V^*$ its polar cone (see~\eqref{eq:def_polar_cone}).
We also define $\tw\gamma = \{(x,y)\in V^2$; $x-y \in \gamma\}$.
Let $q_i \cl V^2 \to V$, $i=1,2$, be the projection to the $i^{th}$ factor.
The following functors are introduced in~\cite{KS90}:
\begin{alignat}{2}
P_\gamma &\cl \Derb(\cor_V) \to \Derb(\cor_V), & \quad
F & \mapsto \roim{q_2}( \cor_{\tw\gamma} \tens \opb{q_1} F) , \\
Q_\gamma &\cl \Derb(\cor_V) \to \Derb(\cor_V), & \quad
F & \mapsto \reim{q_2}( \rhom(\cor_{\tw\gamma^a} , \epb{q_1} F)) .
\end{alignat}
For $\gamma = \{0\}$ we have $P_{\{0\}}(F) \simeq Q_{\{0\}}(F) \simeq F$.
Hence the inclusion $\{0\} \subset \gamma$ induces morphisms of
functors $u_\gamma\cl P_\gamma \to \id$ and $v_\gamma\cl \id \to Q_\gamma$.
If $F$ has compact support, Corollary~\ref{cor:opboim} and
Theorem~\ref{th:opboim}~(ii) give, for any $x\in V$ and using the
identification $T^*V = V\times V^* = V\times T^*_xV$,
\begin{align}
\label{eq:SSPgamF}
\SSi(P_\gamma(F)) \cap T^*_xV
\subset q_2(\SSi(F) \cap \SSi(\cor_{x+\gamma})^a),  \\
\label{eq:SSQgamF}
\SSi(Q_\gamma(F)) \cap T^*_xV
\subset q_2(\SSi(F) \cap \SSi(\cor_{x+\gamma^a})),  
\end{align}
where $q_2 \cl T^*V \to T^*_xV$ is the projection.

For a given subset $\Omega$ of $T^*V$, a morphism $a\cl F \to G$ in
$\Derb(\cor_V)$ is said to be an isomorphism on $\Omega$ if
$\SSi(C(a)) \cap \Omega = \emptyset$, where $C(a)$ is given by the
distinguished triangle $F\to[a] G \to C(a) \to[+1]$.

\begin{proposition}[see~\cite{KS90} Prop. 5.2.3 and Lem. 6.1.5]
\label{prop:cut-off}
We assume that $\gamma$ is proper and $\Int(\gamma) \not= \emptyset$.
For any $F\in \Derb(\cor_V)$ we have
$\SSi(P_\gamma(F))\cup \SSi(Q_\gamma(F)) \subset V \times \gamma^{\circ a}$.
Moreover the morphisms $u_\gamma(F)\cl P_\gamma(F) \to F$ and
$v_\gamma(F) \cl F \to Q_\gamma(F)$ are isomorphisms
on $V \times \Int(\gamma^{\circ a})$.
\end{proposition}

In order to obtain a local statement from Proposition~\ref{prop:cut-off} we will
use Lemmas~\ref{lem:cut-off_split} and~\ref{lem:cutoff_geom} below.
Let $\gamma \subset V$ be a closed convex proper cone.
For $x\in V$ we define $S_x^\gamma \subset T^*V$ by
\begin{equation}\label{eq:def_Sgammax}
  \begin{split}
S_x^\gamma  & = (\dot\SSi(\cor_{x + \gamma})^a \cup \dot\SSi(\cor_{x+\gamma^a}) )
 \setminus (\{x\}\times \Int(\gamma^{\circ a})) \\
& = (\dot\SSi(\cor_{x + \gamma})^a \cup \dot\SSi(\cor_{x+\gamma^a}) )
\cap (V\times \partial\gamma^{\circ a}),
\end{split}
\end{equation}
where the second equality follows from Example~\ref{ex:microsupport}.
Hence we have
\begin{equation}\label{eq:SScorxgam}
\SSi(\cor_{x+\gamma})^a \cup \SSi(\cor_{x+\gamma^a}) \subset
S_x^\gamma \cup (\{x\}\times \Int(\gamma^{\circ a}) ) .
\end{equation}

\begin{lemma}\label{lem:cut-off_split}
Let $F\in \Derb(\cor_V)$ be such that $\supp(F)$ is compact
and let $W\subset V$ be an open subset such that, for any $x\in W$
\begin{equation}\label{eq:hypSSF-split-ponct}
  S_x^\gamma \cap  \dot\SSi(F) = \emptyset.
\end{equation}
Then $v_\gamma(F) \circ u_\gamma(F)|_W \cl P_\gamma(F)|_W \to Q_\gamma(F)|_W$ is an
isomorphism on $\dT^*W$ and we have a distinguished triangle in $\Derb(\cor_W)$
\begin{equation}\label{eq:cut-off_split}
P_\gamma(F)|_W \oplus G \to F|_W \to L \to[+1],
\end{equation}
where $L, G \in \Derb(\cor_W)$ satisfy $\dot\SSi(L) =\emptyset$ and
$\dot\SSi(G)\cap (W \times \gamma^{\circ a}) = \emptyset$.
\end{lemma}
\begin{proof}
(i) We define $L\in \Derb(\cor_W)$  by the distinguished triangle
\begin{equation}\label{eq:cut-off_split-pf1}
P_\gamma(F)|_W \to[v_\gamma(F) \circ u_\gamma(F)] Q_\gamma(F)|_W
 \to[a] L \to[+1].
\end{equation}
The formulas~\eqref{eq:SSPgamF}, \eqref{eq:SSQgamF}, \eqref{eq:SScorxgam}
and~\eqref{eq:hypSSF-split-ponct} give
$\dot\SSi(L) \subset W \times \Int(\gamma^{\circ a})$.
On the other hand Proposition~\ref{prop:cut-off} 
implies that $v_\gamma(F) \circ u_\gamma(F)$ is an isomorphism 
on $W \times \Int(\gamma^{\circ a})$.
Hence $\SSi(L) \cap (W \times \Int(\gamma^{\circ a})) = \emptyset$ and we find
$\dot\SSi(L) = \emptyset$. This proves the first assertion.

\medskip\noindent
(ii) We define $G$ and $F'$ by the distinguished triangles
\begin{gather}
\label{eq:cut-off_split-pf2}
  G \to F|_W \to[v_\gamma(F)] Q_\gamma(F)|_W \to[b] G[1] , \\
\label{eq:cut-off_split-pf3}
F' \to  F|_W \to[a \circ v_\gamma(F)] L \to[+1] .
\end{gather}
Since $v_\gamma(F)$ is an isomorphism on $V \times \Int(\gamma^{\circ a})$ we have
$\dot\SSi(G)\cap (W \times \gamma^{\circ a}) = \emptyset$.
By the octahedral axiom we deduce
from~\eqref{eq:cut-off_split-pf1}-\eqref{eq:cut-off_split-pf3}
the distinguished triangle
$$
G  \to F' \to P_\gamma(F)|_W \to[c] G[1],
$$
where $c = b \circ (v_\gamma(F) \circ u_\gamma(F))$.  By the
triangle~\eqref{eq:cut-off_split-pf2} we have $b \circ v_\gamma(F) =0$, hence
$c=0$. It follow that $F'\simeq  G \oplus  P_\gamma(F)|_W$
and~\eqref{eq:cut-off_split-pf3} gives~\eqref{eq:cut-off_split}.
\end{proof}

Now we write $V = V' \times \R$, where $V'=\R^{n-1}$. We take coordinates
$x=(x',x_n)$ on $V$ and we endow $V'$ and $V$ with the natural Euclidean
structure.  For $c>0$ we let $\gamma_c \subset V$ be the cone
\begin{equation}\label{eq:defgammac}
\gamma_c = \{ (x',x_n) \in V;\; x_n \leq - c \| x'\| \}.
\end{equation}
For a conic subset $C \subset V^*$ and $\varepsilon>0$ we define
$C\langle\varepsilon\rangle  \subset V^*$ by
\begin{equation}\label{eq:defCeps}
C\langle\varepsilon\rangle = \{ \xi \in V^*; \;
 d(\xi, C) < \varepsilon \| \xi\| \}  .
\end{equation}
For real numbers $c,c',\delta,\varepsilon$ and $r\in {}]0,+\infty]$ such that
$c'>c>0$, $\delta\geq 0$ and $\varepsilon>0$ we define a locally closed subset
$Z_c^{r,\delta}$ of $V$ and a subset $W_{c,c'}^{r,\delta,\varepsilon}$ of
$Z_c^{r,\delta}$ by
\begin{align}
\label{eq:defZcrd}
Z_c^{r,\delta} & = (B_r^{V'} \times \R) \setminus ( ((0,-\delta) + \gamma_c)
\cup ((0,\delta) + \Int \gamma_c^a) ) ,  \\
\label{eq:defWcreps}
W_{c,c'}^{r,\delta,\varepsilon} & = \{x\in Z_c^{r,\delta};\;
\{(0,\delta), (0,-\delta)\} \cap \ol{(x+Z_{c'}^{\infty,0})}  = \emptyset, \\
\notag
& \hspace{1.6cm}
I \eqdot \partial Z_c^{\infty,\delta} \cap \partial (x+Z_{c'}^{\infty,0})
\subset B_r^{V'} \times \R \text{ and} \\
\notag
& \hspace{1cm}
\forall y\in I,\, (T^*_yV \cap (S_x^{\gamma_{c'}})^a) \subset
 (T^*_yV \cap \SSi(\cor_{Z_c^{r,\delta}})) \langle\varepsilon\rangle \; \} .
\end{align}

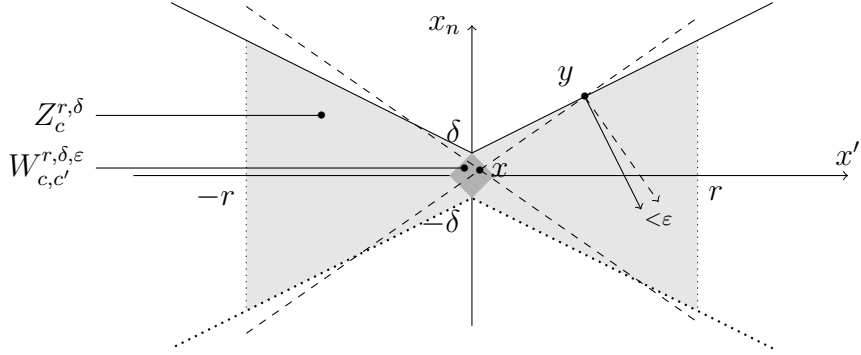
\newcommand{\cc}{.5}
\newcommand{\ccp}{.7}
\newcommand{\bbp}{.15}
\newcommand{\bcp}{0}
\newcommand{\ddelta}{.3}
\newcommand{\xyy}{1.5} 
\newcommand{\grr}{4}
\newcommand{\rr}{3}
\newcommand{\ppp}{$\scriptscriptstyle\bullet$}
\begin{tikzpicture}
\coordinate (A) at (-\grr, \cc*\grr +\ddelta) ;
\coordinate (B) at (0,\ddelta) ; 
\coordinate (C) at (\grr, \cc*\grr +\ddelta) ;
\coordinate (D) at (\grr, -\cc*\grr -\ddelta) ; 
\coordinate (E) at (0,-\ddelta) ; 
\coordinate (F) at (-\grr, -\cc*\grr -\ddelta) ; 
\coordinate (a) at (-\rr, \cc*\rr +\ddelta) ; 
\coordinate (c) at (\rr,\cc*\rr +\ddelta) ;
\coordinate (d) at (\rr, -\cc*\rr -\ddelta) ; 
\coordinate (f) at (-\rr, -\cc*\rr -\ddelta) ; 

\fill [fill=gray!20] (a) -- (B) -- (c) -- (d) -- (E) -- (f) -- cycle;
\draw (A) -- (B)  node [above left] {$\delta$}  -- (C) ;
\draw [thick, dotted] (D) -- (E)  node [below left] {$-\delta$}  -- (F);
\draw [dotted]  (a) -- (f); \draw [dotted] (c) -- (d);

\draw (-\rr,0) node [below left] {$-r$} ;
\draw (\rr,0) node [below right] {$r$} ;
\draw [->]  (0,-2) -- (0,2) node[left] {$x_n$} ;
\draw [->]  (-4.5,0) -- (5,0) node[above] {$x'$} ;

\fill [fill=gray!60] (-\ddelta,0) -- (0,\ddelta) -- (\ddelta,0)
 -- (0,-\ddelta) -- cycle;

\draw [dashed] (-\rr, \ccp*\rr + \bbp) -- (\rr, -\ccp*\rr + \bbp) ;
\draw [dashed]  (-\rr, -\ccp*\rr + \bcp) -- (\rr, \ccp*\rr + \bcp) ;

\draw (\bbp/\ccp/2-\bcp/\ccp/2, \bbp/2+\bcp/2) node {\ppp} node [right] {$x$};
\draw (\xyy, \ccp*\xyy) node {\ppp} node [above left] {$y$};
\draw [->] (\xyy, \ccp*\xyy) -- (\xyy + 1.5*\cc, \ccp*\xyy -1.5) ;
\draw [dashed, ->] (\xyy, \ccp*\xyy) -- (\xyy + 1.4*\ccp, \ccp*\xyy -1.4)
node [below] {$\scriptstyle <\varepsilon$} ;

\draw (-5,.8) node [left] {$Z_c^{r,\delta}$} -- (-2,.8) node {\ppp} ;
\draw (-5,.1) node [left] {$W_{c,c'}^{r,\delta,\varepsilon}$}
 -- (-.1,.1) node {\ppp} ;
\end{tikzpicture}

\begin{lemma}\label{lem:Wnonvide}
(i) For any given $c,c',r,\delta,\varepsilon >0$ the subset
$W_{c,c'}^{r,\delta,\varepsilon}$ is open in $V$. \\
(ii) For any given $c,r,\varepsilon >0$, setting $c' = c + \varepsilon/2$ and
choosing $\delta < \varepsilon r/2$, we have
$W_{c,c'}^{r,\delta,\varepsilon} \not=\emptyset$.
\end{lemma}
\begin{proof}
The first assertion is clear on the definition~\eqref{eq:defWcreps}.  For the
second claim we check easily that$(0,0) \in W_{c,c'}^{r,\delta,\varepsilon}$.
\end{proof}

\begin{lemma}\label{lem:cutoff_geom}
Let $c_2 > c_1 >0$ and $r>0$ be given.  We choose $c, \varepsilon >0$ such
that, using the notation~\eqref{eq:defCeps}, we have
$(\partial \gamma_c^{\circ})\langle\varepsilon\rangle 
\subset \gamma_{c_2}^{\circ} \setminus \Int(\gamma_{c_1}^{\circ})$.
We choose $\delta>0$ and $c'>c$. We set $U = B_r^{V'} \times \R$.
Then, for all $F\in \Derb(\cor_U)$ satisfying
\begin{equation}\label{eq:hypSSF_split}
\dot\SSi(F) \cap \bigr( U \times 
(\gamma_{c_2}^{\circ a} \setminus \Int(\gamma_{c_1}^{\circ a})) \bigl) = \emptyset,
\end{equation}
we have
$S_x^{\gamma_{c'}} \cap  \dot\SSi(F\tens \cor_{Z_c^{r,\delta}}) = \emptyset$,
for all $x\in W_{c,c'}^{r,\delta,\varepsilon}$.
\end{lemma}
\begin{proof}
Let $x\in W_{c,c'}^{r,\delta,\varepsilon}$ be a given point. Let us assume that
there exists a point
$(y;\eta) \in S_x^{\gamma_{c'}} \cap  \dot\SSi(F\tens \cor_{Z_c^{r,\delta}})$.
We set
$$
I = \partial Z_c^{\infty,\delta} \cap \partial (x+Z_{c'}^{\infty,0})
= \partial Z_c^{\infty,\delta} \cap \pi_V(S_x^{\gamma_{c'}})
$$
as in~\eqref{eq:defWcreps}.

\medskip\noindent
(i) We first prove that we must have $y\in I$.
Since $y\in \pi_V(S_x^{\gamma_{c'}})$ we have to check that
$y\in \partial Z_{c}^{r,\delta}$.
We remark that the set of $y'\in V$ such that $(y';\eta) \in S_x^{\gamma_{c'}}$
is a line through $x$. This line meets $\partial{(x+Z_{c'}^{\infty,0})}$ at some
point $y_1$, necessarily distinct from $(0,\pm \delta)$.
Then the definition of $W_{c,c'}^{r,\delta,\varepsilon}$ in~\eqref{eq:defWcreps}
implies $-\eta \in (T^*_{y_1}V \cap \SSi(\cor_{Z_c^{r,\delta}}))
\langle\varepsilon\rangle$.
Since $y_1 \not= (0, \pm \delta)$ we deduce
$\eta \in (\partial \gamma_c^{\circ a})\langle\varepsilon\rangle 
\subset \gamma_{c_2}^{\circ a} \setminus \Int(\gamma_{c_1}^{\circ a})$.
By~\eqref{eq:hypSSF_split} we have in particular
\begin{equation}\label{eq:pflemcutoff1}
(y;\eta) \not\in \SSi(F).
\end{equation}
This implies that $ T^*_yV \cap \SSi(\cor_{Z_c^{r,\delta}}) \not=0$.
We deduce $y\in \partial Z_{c}^{r,\delta}$ and then $y\in I$.

\medskip\noindent
(ii) Since $y\in \pi_V(S_x^{\gamma_{c'}})$ we have $y \not= (0, \pm \delta)$,
hence $T^*_yV \cap \SSi(\cor_{Z_c^{r,\delta}}) \subset  \partial \gamma_c^{\circ}
\subset \gamma_{c_2}^{\circ} \setminus \Int(\gamma_{c_1}^{\circ})$.
We deduce from~\eqref{eq:hypSSF_split} that
$T^*_yV \cap \SSi(F) \cap \SSi(\cor_{Z_c^{r,\delta}})^a$ is contained in the
zero section.  Then Corollary~\ref{cor:opboim} gives
\begin{equation*}
T^*_yV \cap \SSi(F\tens \cor_{Z_c^{r,\delta}}) \subset T^*_yV \cap (\SSi(F) +
\SSi(\cor_{Z_c^{r,\delta}})).
\end{equation*}
Hence $\eta$ can be written $\eta = \eta_1 +\eta_2$, with
$\eta_1 \in T^*_yV \cap \SSi(F)$ and 
$\eta_2 \in T^*_yV \cap \SSi(\cor_{Z_c^{r,\delta}})$.
By~\eqref{eq:pflemcutoff1} we have $\eta_2\not=0$. Since
$ \partial Z_{c}^{r,\delta}$ is a smooth hypersurface around $y$, we have
in fact $T^*_yV \cap \SSi(\cor_{Z_c^{r,\delta}}) = \R_{\geq 0}\eta_2$.
As in~(i) the definition of $W_{c,c'}^{r,\delta,\varepsilon}$ gives
$-\eta \in (\R_{\geq 0}\eta_2)\langle\varepsilon\rangle$.
Since $(\R_{\geq 0}\eta_2)\langle\varepsilon\rangle$ is convex, it follows that
$$
-\eta_1 = -\eta+\eta_2 \in (\R_{\geq 0}\eta_2)\langle\varepsilon\rangle
\subset (\partial \gamma_c^{\circ})\langle\varepsilon\rangle
\subset \gamma_{c_2}^{\circ} \setminus \Int(\gamma_{c_1}^{\circ}) .
$$
This contradicts~\eqref{eq:hypSSF_split} and proves the lemma.
\end{proof}

\begin{proposition}\label{prop:cut-off_split_local}
Let $c_2 > c_1 >0$ and $r>0$ be given.
Then there exists $r_1 >0$ such that, for any $s\in\R$, we have,
setting $U = B_r^{V'} \times \R$ and $B = B_{(0,s),r_1}^V$:
for all $F\in \Derb(\cor_U)$ satisfying~\eqref{eq:hypSSF_split},
there exist $F_1,F_2,L \in \Derb(\cor_B)$  and a
distinguished triangle in $\Derb(\cor_B)$
\begin{equation*}
F_1 \oplus F_2 \to F|_B \to L \to[+1]
\end{equation*}
such that $\dot\SSi(F_1) = \dot\SSi(F) \cap (B \times \gamma_{c_1}^{\circ a})$,
$\dot\SSi(F_2) = \dot\SSi(F) \setminus \dot\SSi(F_1)$ and
$\SSi(L) \subset T^*_BB$.
In particular $F_1 \to F|_B$ is an isomorphism on
$B \times \gamma_{c_1}^{\circ a}$ and $F_2 \to F|_B$ is an isomorphism on
$B \times (V^* \setminus \Int(\gamma_{c_2}^{\circ a}))$.
\end{proposition}
\begin{proof}
Since the statement is invariant by translation in the factor $\R$
of $V = V' \times \R$ we can assume $s=0$.
We choose $c, \varepsilon$ such that
$(\partial \gamma_c^{\circ})\langle\varepsilon\rangle 
\subset \gamma_{c_2}^{\circ} \setminus \Int(\gamma_{c_1}^{\circ})$,
as in Lemma~\ref{lem:cutoff_geom}. We set $c' = c+\varepsilon/2$.
We remark that $c_1<c'<c_2$.
By Lemma~\ref{lem:Wnonvide} we can choose $\delta>0$
so that $W_{c,c'}^{r,\delta,\varepsilon} \not=\emptyset$.
We choose $r_1$ such that
$B \eqdot B_{(0,0),r_1}^V \subset W_{c,c'}^{r,\delta,\varepsilon}$.

By Lemma~\ref{lem:cutoff_geom} we can apply Lemma~\ref{lem:cut-off_split} to
$F\tens \cor_{Z_c^{r,\delta}}$ with $\gamma=\gamma_{c'}$. We obtain the
distinguished triangle~\eqref{eq:cut-off_split} with $F$ replaced by
$F\tens \cor_{Z_c^{r,\delta}}$.
Then we set $F_1 = P_{\gamma_{c'}}(F\tens \cor_{Z_c^{r,\delta}})|_B$ and
$F_2 = G|_B$. The proposition follows.
\end{proof}

\section{Non constant groups of sections}
\label{sec:noncstsections}

In this section we use Proposition~\ref{prop:cut-off_split_local} to give
conditions on a sheaf which imply that it has non isomorphic cohomology groups
over some open sets. These open sets only depend on the microsupport. This will
be used in the proof of the Gromov-Eliashberg Theorem to insure the non
triviality of the microsupport of a sheaf (see
Proposition~\ref{prop:sections-fixed-open}).

We still use the notations $V'=\R^{n-1}$ and $V = V' \times \R$ of
Section~\ref{sec:cutoff}.  We also let $\gamma_c \subset V$ be the cone defined
in~\eqref{eq:defgammac} for any $c>0$.

\begin{proposition}\label{prop:camellight}
Let $c_2 > c_1 >0$ and $r>0$ be given and let $r_1 >0$ 
be given by Proposition~\ref{prop:cut-off_split_local}.
We set $U = B_r^{V'} \times \R$. Let $F\in \Derb(\cor_U)$ and
$$
S_1 = \dot\SSi(F) \cap (U \times \gamma_{c_1}^{\circ a}) ,\qquad
S_2 = \dot\SSi(F) \setminus S_1.
$$
We assume that $F$ satisfies~\eqref{eq:hypSSF_split} and that
there exist $a<b$ such that
\begin{itemize}
\item [(i)] $F|_{\{0\}\times\R} \simeq \cor_{[a,b[} \oplus F'$ for some
 $F'\in \Derb(\cor_\R)$,
\item [(ii)] $S_i \cap T^*_{(0,a)}V = \emptyset$ or
$S_i \cap T^*_{(0,b)}V = \emptyset$, for $i=1$ and $i=2$.
\end{itemize}
Then $b-a \geq r_1$.
\end{proposition}
\begin{proof}
If $b-a < r_1$, the ball $B$ of center $(0,(a+b)/2)$ and radius
$r_1$ contains $a$ and $b$. Proposition~\ref{prop:cut-off_split_local}
implies that we have a distinguished triangle $F_1 \oplus F_2 \to F|_B \to L
\to[+1]$, where $\SSi(L) \subset T^*_BB$ and $\SSi(F_i) \subset S_i$, $i=1,2$.

Then the hypothesis~(i) and Lemma~\ref{lem:deux_decompositions} implies
that $\cor_{[a,b[}$ is a direct summand of $F'|_{B \cap (\{0\}\times\R)}$ where
$F'=F_1$ or $F'=F_2$. Hence $\SSi(F')$ meets both $T^*_{(0,a)}V$ and
$T^*_{(0,b)}V$ outside the zero-section, which contradicts the hypothesis~(ii).
\end{proof}

Proposition~\ref{prop:camellight} will be used together with
Lemma~\ref{lem:ext_sect} below. It says that we can extend sections over a
segment to some neighborhood of this segment which only depends on the
microsupport. For given $c,r,s>0$ with $r<s/c$, we define $U_{r,s}^c \subset V$
by
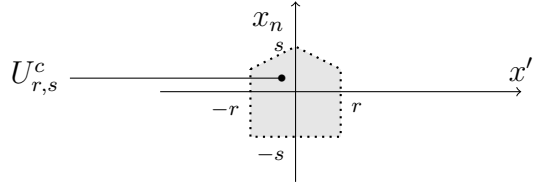
\begin{equation}\label{eq:defUceta}
U_{r,s}^c = \{ (x',x_n) \in V;\;
\| x'\| < r \text{ and }  -s < x_n < s - c \| x'\|  \}.
\end{equation}

\newcommand{\ccc}{.5}
\newcommand{\eeta}{.6}
$$
\begin{tikzpicture}
\coordinate (A) at (-\eeta, -\ccc*\eeta +\eeta); \coordinate (B) at (0,\eeta);
\coordinate (C) at (\eeta, -\ccc*\eeta +\eeta);
\coordinate (D) at (\eeta, -\eeta) ; \coordinate (E) at (-\eeta, -\eeta);

\fill [fill=gray!20] (A) -- (B) -- (C) -- (D) -- (E) -- cycle;
\draw [thick, dotted] (A) -- (B) node[left] {$\scriptstyle s$} -- (C)
 -- (D) -- (E) -- cycle;

\draw (0,-\eeta) node[below left] {$\scriptstyle -s$} ;
\draw (-\eeta,0) node[below left] {$\scriptstyle -r$} ;
\draw (\eeta,0) node[below right] {$\scriptstyle r$} ;

\draw [->] (-3*\eeta,0) -- (5*\eeta,0) node[above] {$x'$} ;
\draw [->]  (0,-2*\eeta) -- (0,2*\eeta) node[below left] {$x_n$} ;
\draw (-3,.3*\eeta) node [left] {$U_{r,s}^c$}
 -- (-.3*\eeta,.3*\eeta) node {\ppp} ;
\end{tikzpicture}
$$
The Euclidean structure on $V'$ gives an identification $V'\simeq V'^*$ and we
can define
$T^{*,out}V' = \{ (x';\xi') \in T^*V'$; $\xi'=\lambda\, x'$, $\lambda \geq 0\}$.
We also set $C_r^\pm = (B_r^{V'}\setminus \{0\})\times \{ \pm x_n \geq 0\}$.
By Example~\ref{ex:microsupport} we have
\begin{align}
\label{eq:boundUscpos}
&\SSi(\cor_{U_{r,s}^c}) \cap \opb{\pi_V}(C_r^+)
\subset \partial\gamma_c^{\circ a} \cap (T^{*,out}V' \times T^*\R),\\
\label{eq:boundUscneg}
&\SSi(\cor_{U_{r,s}^c}) \cap \opb{\pi_V}(C_r^-)
\subset T^*_{V'}V' \times \{(-s;\sigma);\; \sigma\geq 0\}.
\end{align}

\begin{lemma}\label{lem:ext_sect}
Let $c,r,s>0$ be given with $r<s/c$. Let $U\subset V$ be a neighborhood of
$\ol{U_{r,s}^c}$ and let $F\in \Derb(\cor_U)$ be such that
$\SSi(F) \subset U \times \Int(\gamma_c^{\circ a})$.
Then the morphism  $\rsect(U_{r,s}^c;F) \to \rsect(]-s,s[; F|_{\{0\}\times ]-2s,2s[})$
is an isomorphism.
\end{lemma}
\begin{proof}
We choose $R>r$ such that $\ol{U_{R,s}^c} \subset U$ and we define
$G \in \Derb(\cor_{V'})$ by $G = \roim{q_1}(\rsect_{U_{R,s}^c}(F))$, where
$q_1 \cl V\to V'$ is the projection.  By~\eqref{eq:boundUscpos},
\eqref{eq:boundUscneg} and the hypothesis on $\SSi(F)$ we obtain that
$\SSi(F) \cap \SSi(\cor_{U_{R,s}^c})$ is contained in the zero section over
$C_R^+$ and $C_R^-$.  By Corollary~\ref{cor:opboim} and
Theorem~\ref{th:opboim}~(ii) we deduce
\begin{align}
\notag
\SSi(G) \cap \opb{\pi_{V'}}(B_R^{V'} \setminus \{0\})  \subset
\{(x';\xi');\; & \exists (x',x_n;\xi_1',\xi_{1,n}) \in \SSi(\cor_{U_{R,s}^c})^a , \\
\label{eq:ext_sect1}
& \exists (x',x_n;\xi_2',\xi_{2,n})  \in \SSi(F), \\
\notag
& \xi' = \xi_1' +\xi_2',\,  \xi_{1,n} +\xi_{2,n} =0 \} .
\end{align}
If $x_n\geq 0$ in~\eqref{eq:ext_sect1}, then~\eqref{eq:boundUscpos} gives
$\xi_{1,n} = -c^{-1} \|\xi_1' \|$ and
$\xi_1' = \lambda\, x'$, for some $\lambda\leq 0$.
The hypothesis on $\SSi(F)$ gives $\xi_{2,n} \geq c^{-1} \|\xi_2' \|$.
We deduce $\|\xi_1' \| \geq \|\xi_2' \|$ and then
$\langle x', \xi' \rangle \leq \lambda \|x'\|^2 + \|x'\|\, \|\xi_2' \| \leq 0$.
\\
If $x_n\geq 0$ in~\eqref{eq:ext_sect1}, then~\eqref{eq:boundUscneg} and the
hypothesis on $\SSi(F)$ give $\xi'=0$.

\smallskip
We conclude that $\SSi(G) \cap T^*(V' \setminus \{0\})  \subset
 \{ (x';\xi') \in T^*V'$; $\langle x', \xi' \rangle \leq 0\}$.
By Corollary~\ref{cor:Morse} we deduce that
$\rsect(B_{r_2}^{V'};G) \to \rsect(B_{r_1}^{V'};G)$ is an isomorphism for any
$0<r_1 < r_2 < R$. For $r_2=r$ and $r_1\to 0$ we find the isomorphism of the
lemma.
\end{proof}

\begin{proposition}\label{prop:sections-fixed-open}
Let $c_2 > c_1 >0$ and $r>0$ be given.
We set $U = B_r^{V'} \times \R$. 
Then there exist non empty connected open subsets $W_i$, $i=1,\ldots,4$, of $U$
such that $\ol{W_i} \subset W_{i+1}$, for $i=1,\ldots,3$, which satisfy the
following.  For any $F\in \Derb(\cor_U)$ such that
\begin{itemize}
\item [(i)] $F$ satisfies~\eqref{eq:hypSSF_split},
\item [(ii)] $\Lambda= \dot\SSi(F)$ is a Lagrangian submanifold of $\dT^*U$ and
  $F$ is simple along $\Lambda$,
\item [(iii)] setting
  $\Lambda_1 = \Lambda \cap (U \times \gamma_{c_1}^{\circ a})$, the natural map
  $\Lambda_1/\R_{>0} \to B_r^{V'}$ is proper and of degree $1$,
\end{itemize}
there exists $F_1 \in \Mod(\cor_{W_4})$ such that
\begin{itemize}
\item [(a)] $\dot\SSi(F_1) \subset T^*W_4 \cap T^*_{\sigma>0}(V'\times\R)$
and $\dot\SSi(F_1) \subset  \Conv( \Lambda_1 )$,
\item [(b)] there exists $u\in \sect(W_4;F_1)$ such that $u|_{W_3} \not=0$ and
  $u|_{W_2} =0$ or there exists $v\in \sect(W_2;F_1)$ such that $v|_{W_1}
  \not=0$ and $v$ is not in the image of $\sect(W_3;F_1) \to \sect(W_2;F_1)$.
\end{itemize}
\end{proposition}
\begin{remark}
The conclusion~(b) of the proposition also holds with $\sect(W_i;F_1)$ replaced
by $H^0(W_i;F')$, where $F'$ is obtained from $F$ by a shift in the derived
category and a vertical translation in $V=V'\times\R$ (see part~(C) of the
proof).  We introduce $F_1$ because we want to be sure to stay in the bounded
derived category when we use this result in the paragraph~\ref{sec:pf4}.  It is
likely that the properties of the microsupport hold for unbounded complexes. In
this case the introduction of $F_1$ would be useless.
\end{remark}

\begin{proof}
(A) Let $r_1$ be given from $r,c_1,c_2$ by
Proposition~\ref{prop:cut-off_split_local}.
We set $D_1 = \mo{]}-r_1/8,-r_1/16[$,
$D_2 = \mo{]}-r_1/4,0[$,
$D_3 = \mo{]}-r_1/2,r_1/2[$,
$D_4 = \mo{]}-r_1,r_1[$.
For any $x\in V'$ and $i=1,\ldots,4$, we set $D_{i,x} = \{x\} \times D_i$.
We choose $t_i$, $t'_i>0$ and $0<\rho < \min\{t'_i/c_1\}$ 
such that $W_{i,x} \eqdot (x,t_i) + U_{\rho,t'_i}^{c_1}$ satisfies 
$W_{i,x} \cap \{x\}\times\R = D_{i,x}$.
Now we can choose $\varepsilon>0$ and non empty open subsets $W_i$,
$i=1,\ldots,4$, of $U$ such that $\ol{W_i} \subset W_{i+1}$ and,
for all $\|x\|<\varepsilon$,
\begin{equation}\label{eq:cor-sect-fix-open}
D_{1,x} \subset W_1 \subset W_2 \subset W_{2,x} , \qquad
D_{3,x} \subset W_3 \subset W_4 \subset W_{4,x}.
\end{equation}

\medskip\noindent
(B) We can find $x \in V'$ arbitrarily close to $0$ such that
$\Lambda$ intersects $T^*_xV'\times T^*\R$ transversally.  Since
$\Lambda_1/\R_{>0} \to B_r^{V'}$ is proper, $\Lambda_1 \cap (T^*_xV'\times T^*\R)$
consists of finitely many half lines, all contained in $T^*_xV'\times T^*I$ for
some bounded interval $I=\mo{]}a,b[$.
Then $\Lambda \cap (T^*_xV'\times T^*I)$ is also a finite set of half lines and
we write
$$
\Lambda \cap (T^*_xV' \times T^*I) = \bigsqcup_{k=1}^l
\{(x, s_k;\sigma\xi_k,\sigma); \; \sigma \not= 0 \}.
$$
Near a point $(x,s_i;\xi_i,1)$ the Lagrangian $\Lambda$ is the conormal bundle
to a smooth hypersurface, say $X_i$. If $s_i=s_j$, we have $\xi_i\not= \xi_j$
and the hypersurfaces $X_i$ and $X_j$ are transversal at $(x,s_i)$. Then, by
moving $x$ a little bit we can make $s_i$ and $s_j$ distinct, which we will
assume from now on (for all pairs $i\not=j$).
We set $D =\{x\} \times I \subset V'\times I$,
$S = \pi_{V'\times \R}(\Lambda) \cap D = \{s_1,\ldots, s_l \}$
and $S_1 = \pi_{V'\times \R}(\Lambda_1) \cap D$.

\medskip\noindent
(C) We remark that $F|_D$ is simple.  By Proposition~\ref{prop:simplesheafR} it
follows that $F|_D$ is a sum of sheaves of the type $\cor_{]a,s_i[}$,
$\cor_{[s_i,b[}$ or $\cor_{[s_i,s_j[}$, up to shift, where $s_i,s_j\in S$.
Since $\Lambda_1 \to B_{r}^{V^2}$ has degree $1$, the set $S_1$ is of odd
order. Hence there exist $s \in S_1$ and $d\in\Z$ such that $F|_D[d]$ has a
direct summand isomorphic to $\cor_{]a,s[}$, $\cor_{[s,b[}$, $\cor_{[s,s'[}$ or
$\cor_{[s',s[}$, with $s'\in S\setminus S_1$.

In case the summand is $\cor_{[s,s'[}$ or $\cor_{[s',s[}$,
Proposition~\ref{prop:camellight} gives $|s'-s| \geq r_1$ (recall that $r_1$ is
given by Proposition~\ref{prop:cut-off_split_local}).
We could also have assumed in~(B) that $S_1 \subset \mo{]}a+r_1,b-r_1[$.
Hence in any case we obtain the following.
We set $F' = \opb{T_s}F[d]$, where $T_s\cl V'\times\R \to V'\times\R$ is the
translation $(v',t) \mapsto (v',t+s)$. Then $F'|_{D_{4,x}}$ has
a direct summand isomorphic to $\cor_{ ]-r_1,0[}$ or $\cor_{ [0,r_1[}$.
We have $(x,0) \in \pi_{V'\times \R}(\Lambda_1)$
and  $(x,0) \not\in \pi_{V^2\times \R}(\Lambda \setminus \Lambda_1)$.

\medskip\noindent
(D) Let us set $B = B_{r_1/2}^{V^2\times\R}$. For $x$ close enough to $0$
we have $B\subset B_{(x,0),r_1}^{V^2\times\R}$. By
Proposition~\ref{prop:cut-off_split_local} we have a distinguished triangle
$$
F'_1 \oplus F'_2 \to F'|_B \to L \to[+1],
$$
where $F'_1,F'_2,L \in \Derb(\cor_B)$ satisfy
$\SSi(L) \subset T^*_BB$, $\dot\SSi(F'_1) = \Lambda_1 \cap T^*B$
and $\dot\SSi(F'_2) \cap \Lambda_1 =\emptyset$.
By Lemma~\ref{lem:deux_decompositions} we obtain that $F'_1|_{D_{4,x}}$ or
$F'_2|_{D_{4,x}}$ has a direct summand isomorphic to $\cor_{ ]-r_1,0[}$ or
$\cor_{ [0,r_1[}$.
Since $(x,0) \not\in \pi_{V^2\times \R}(\Lambda \setminus \Lambda_1)$ we deduce
that $F'_1|_{D_{4,x}}$ has a direct summand isomorphic to $\cor_{ [0,r_1[}$ or
$\cor_{ ]-r_1,0[}$.

\medskip\noindent
(E) We define  $F_1 = H^0(F'_1) \in \Mod(\cor_W)$.
Then $F_1|_{D_{4,x}}$ also has a direct summand isomorphic to $\cor_{ [0,r_1[}$
or $\cor_{ ]-r_1,0[}$.
Hence the conclusion~(b) of the corollary holds with
$\sect(W_i;F_1)$ replaced by $\sect(D_{i,x};F_1|_{D_{4,x}})$.
By Lemma~\ref{lem:ext_sect} and the inclusions~\eqref{eq:cor-sect-fix-open}
we deduce that~(b) also holds as stated.

By Proposition~\ref{prop:boundSStronc} we have
$\SSi(F_1) \subset \Conv(\SSi(F'_1))$, which gives the assertion~(a) of the
corollary.
\end{proof}

\section{Quantization}

In this section we recall the main result of~\cite{GKS12}.  Let $N$ be a
manifold and $I$ an open interval of $\R$ containing $0$.  We consider a
homogeneous Hamiltonian isotopy $\Psi\cl \dT^*N \times I \to \dT^*N$ of
class $\Cinf$. For $t\in I$, $p\in \dT^*N$ we set $\Psi_t(p) = \Psi(p,t)$.
Hence $\Psi_0 = \id_{\dT^*N}$ and, for each $t\in I$, $\Psi_t$ is symplectic
diffeomorphism such that $\Psi_t(x;\lambda\xi) = \lambda\cdot \Psi_t(x;\xi)$,
for all $(x;\xi) \in \dT^*N$ and $\lambda>0$.
We let $\Lambda_{\Psi_t} \subset \dT^*N^2$ be the twisted graph of $\Psi_t$.
We can see that there exists a unique conic Lagrangian submanifold
$\Lambda_\Psi \subset \dT^*(N^2 \times I)$ such that 
$\Lambda_{\Psi_t} = i_{t,d} \opb{i_{t,\pi}}(\Lambda_\Psi)$, for all $t\in I$,
where $i_t$ is the embedding $N^2\times \{t\} \to N^2 \times I$.
We let $\Derlb(\cor_{N^2\times I})$ be the full subcategory of
$\Der(\cor_{N^2\times I})$ formed by the $F$ such that $F|_C \in \Derb(\cor_C)$
for all compact subsets $C \subset N^2\times I$.

\begin{theorem}[Theorem.~4.3 of~\cite{GKS12}]\label{thm:GKS}
There exists a unique $K_\Psi \in\Derlb(\cor_{N^2\times I})$ such that
$\dot\SSi(K_\Psi) \subset \Lambda_\Psi$ and
$K_\Psi|_{N^2\times \{0\}} \simeq \cor_\Delta$.
Moreover $K_\Psi$ is simple along $\Lambda_\Psi$ and both projections
$\supp(K_\Psi) \to N\times I$ are proper.
\end{theorem}
The fact that $K$ is simple along $\Lambda_\Psi$ is not explicitly stated
in~\cite{GKS12} but it follows from the claim
$K_\Psi|_{N^2\times \{0\}}\simeq \cor_\Delta$.

We reduce the case of non homogeneous Hamiltonian isotopies to the homogeneous
framework by adding one variable as follows. Let $M$ be a connected manifold
and let $\Phi\cl T^*M\times I \to T^*M$ be a Hamiltonian $\Cinf$ isotopy.
We assume that $\Phi$ has compact support, that is, there exists
a compact subset $C\subset T^*M$ such that $\Phi(p,t) = p$ for all
$p\in T^*M\setminus C$ and all $t\in I$.
We let $(s;\sigma)$ be the coordinates on $T^*\R$ and we recall the map
$\rho_M \cl T^*M\times \dT^*\R \to T^*M$,
$((x;\xi),(s;\sigma)) \mapsto (x;\xi/\sigma)$ defined
in~\eqref{eq:def_rho}. By~\cite[Prop. A.6]{GKS12} there exists a homogeneous
Hamiltonian isotopy $\Psi\cl \dT^*(M\times \R) \times I \to \dT^*(M\times \R)$
whose restriction to $T^*M\times \dT^*\R\times I$ gives the commutative diagram
\begin{equation}\label{eq:diag-PhiPsi}
\begin{tikzcd}[column sep=2cm]
T^*M\times \dT^*\R\times I \rar{\Psi} \dar{\rho_M \times\id_I}
                                 &T^*M\times \dT^*\R \dar{\rho_M}  \\
T^*M\times I\rar{\Phi}              & T^*M \pointdiag 
\end{tikzcd}
\end{equation}
Moreover there exists a $\Cinf$-function $u\cl (T^*M) \times I\to\R$ 
such that
\begin{equation}\label{eq:expr_twPhi}
\Psi((x;\xi),(s;\sigma),t) = ((x';\xi'), (s+ u(x,\xi/\sigma,t);\sigma)),
\end{equation}
where $(x';\xi'/\sigma) = \Phi_t(x;\xi/\sigma)$.
In particular, if $(x;\xi/\sigma) \not\in C$, we have $(x';\xi') = (x;\xi)$.
Since $\Psi_t$ is symplectic, we deduce $d(u|_{T^*M \times\{t\}})=0$
outside $C$. It follows that if $\Omega$ is a connected component of
$T^*M \setminus C$, then there exists $v_\Omega\cl I \to \R$ such that
\begin{equation}\label{eq:expr_twPhi2}
\Psi((x;\xi),(s;\sigma),t) = ((x;\xi), (s+ v_\Omega(t);\sigma)),
\quad \text{for all $(x;\xi) \in \Omega$.}
\end{equation}

\begin{corollary}\label{cor:quantconic}
Let $\Phi\cl T^*M\times I \to T^*M$ be a Hamiltonian $\Cinf$ isotopy with
compact support.  Then, for any $t\in I$, there exist a closed conic connected
Lagrangian submanifold $\Lambda_t \subset \dT^*(M^2\times\R)$ and
$K_{t} \in \Derb(\cor_{M^2 \times \R})$ such that
\begin{itemize}
\item [(i)] $\rho_{M^2}\cl T^*M^2 \times \dT^*\R \to T^*M^2$ induces a
  diffeomorphism between
  $(\Lambda_t \cap (T^*M^2 \times \dT^*\R))/\R^\times$ and
  $\Lambda_{\Phi_t}$, the twisted graph of $\Phi_t$,
\item [(ii)] $\dot\SSi(K_{t}) = \Lambda_t$ and $K_{t}$ is simple
  along $\Lambda_t$.
\end{itemize}
\end{corollary}
\begin{proof}
(i) Let $\Psi$ be the homogeneous Hamiltonian isotopy introduced in the
diagram~\eqref{eq:diag-PhiPsi}.  We see on~\eqref{eq:expr_twPhi} that $\Psi$
preserves the variable $\sigma$, that is, $\Lambda_\Psi$ is contained in
$\Sigma \eqdot \{\sigma+\sigma'=0\}$.
Let us define $q\cl (M\times \R)^2\times I \to M^2 \times \R\times I$,
$(x,s,x',s',t) \mapsto (x,x',s-s',t)$. Then $\Sigma = \operatorname{im} q_d$ 
and the quotient map to the symplectic reduction of $\Sigma$ is $q_\pi$.
Hence we can write $\Lambda_{\Psi} = q_d \opb{q_\pi} (\Lambda)$, where
$\Lambda \subset \dT^*(M^2 \times \R)$ is given by 
$\Lambda = q_\pi \opb{q_d} (\Lambda_{\Psi})$.
Now we set $\Lambda_t = i_{t,d} \opb{i_{t,\pi}}(\Lambda)$, where
$i_t \cl M^2 \times \R\times \{t\} \to M^2 \times \R\times I$ is the
embedding, and~(i) follows from the diagram~\eqref{eq:diag-PhiPsi}.

\medskip\noindent
(ii) Let $K_\Psi \in \Derlb(\cor_{(M\times\R)^2\times I})$ be given by
Theorem~\ref{thm:GKS}. Let us first check that
$K_{\Psi,t} \eqdot K_\Psi|_{(M\times\R)^2\times \{t\}} \in \Derb(\cor_{(M\times\R)^2})$
for any $t\in I$.

Let $C\subset T^*M$ be a compact subset such that $\Phi(p,t) = p$ for $p$
outside $C$. We remark that we may enlarge $C$ so that
$\Omega = T^*M \setminus C$ has a single connected component when $M$ is not the
circle $S^1$, and two components, say $\Omega_\pm$, corresponding to
$\pm \xi \gg0$, when $M = S^1$.  We set $Z = \pi_M(C)$, $U=M\setminus Z$ and
$q_t = q|_{(M\times \R)^2\times \{t\}}$.  By~\eqref{eq:expr_twPhi2} and the
unicity of $K_\Psi|_{(U\times\R)^2\times I}$ we find
$$
K_{\Psi,t}|_{(U\times\R)^2} \simeq
\begin{cases}
 \opb{q_t} \cor_{\Delta_M \times \{v_\Omega(t)\}},
& \hspace{-5mm}\text{if $M \not= S^1$,} \\
\begin{cases}
  \opb{q_t} 
  \cor_{\Delta_M \times \mo{]}v_{\Omega_-}(t), v_{\Omega_+}(t)[} [1], \\
 \text{or } \opb{q_t}
\cor_{\Delta_M \times [v_{\Omega_+}(t), v_{\Omega_-}(t)]} ,
\end{cases}
& \hspace{-5mm}\text{if $M = S^1$.}
\end{cases}
$$
We also have $\Lambda_\Psi \cap \dT^*(Z \times U \times \R^2) = \emptyset$.
Hence $K_\Psi$ is locally constant on $Z\times U \times \R^2 \times I$.
We deduce $K_\Psi|_{Z\times U \times \R^2 \times I} \simeq 0$ since this holds at $t=0$.
In the same way $K_\Psi|_{U\times Z \times \R^2 \times I} \simeq 0$.

We conclude that $K_{\Psi,t}|_{(M^2 \setminus Z^2)\times \R^2}$ is a bounded complex.
Since $Z^2$ is compact we have $K_{\Psi,t} \in \Derb(\cor_{(M\times\R)^2})$ as
claimed.

Now, by~(i) and Corollary~\ref{cor:opbeqv} there exists
$K_t \in \Derb(\cor_{M^2 \times \R})$ such that $K_{\Psi,t} \simeq \opb{q}K_t$.
Then $K_t$ satisfies~(ii).
\end{proof}

Now we apply Proposition~\ref{prop:sections-fixed-open} to the quantization
given by Corollary~\ref{cor:quantconic} in the following situation.
Let $V=\R^n$ be a vector space and let $r,A>0$. We will consider the following
hypothesis on a symplectic map $\varphi \cl T^*V \to T^*V$:
\begin{equation}\label{eq:hyp-part-quant}
\left\{  \hspace{-5mm} \begin{minipage}{11cm}
\begin{itemize}
\item[(i)] there exists a $\Cinf$ Hamiltonian isotopy with compact support,
  $\Phi\cl T^*V\times I \to T^*V$, such that $\varphi = \Phi_1$,
\item[(ii)] $\Lambda_\varphi \cap (B_{r}^{V^2} \times B_{3Ar}^{V^{*2}})
\subset  B_{r}^{V^2} \times B_{2Ar}^{V^{*2}}$,
\item[(iii)] setting $\Lambda_\varphi^1 
  = \Lambda_\varphi \cap (B_{r}^{V^2} \times B_{3Ar}^{V^{*2}})$, the map
  $\Lambda_\varphi^1 \to B_{r}^{V^2}$ induced by the projection to the base is
  of degree $1$.
\end{itemize} \end{minipage} \right.
\end{equation}

\begin{corollary}\label{cor:part-quant}
Let $A,r>0$ be given.
There exist non empty connected open subsets $W_i$, $i=1,\ldots,4$, of
$B_{r}^{V^2}\times\R$ such that $\ol{W_i} \subset W_{i+1}$, for $i=1,\ldots,3$,
which satisfy the following.
For any symplectic map $\varphi \cl T^*V \to T^*V$
satisfying~\eqref{eq:hyp-part-quant}, there exists $L\in \Mod(\cor_{W_4})$
such that
\begin{itemize}
\item [(i)] $\dot\SSi(L) \subset T^*W_4 \cap T^*_{\sigma>0}(V^2\times\R)$
and $\rho_{V^2}(\dot\SSi(L)) \subset  \Conv(\Lambda_\varphi^1)$,
\item [(ii)] there exists $u\in \sect(W_4;L)$ such that $u|_{W_3} \not=0$ and
  $u|_{W_2} =0$ or there exists $v\in \sect(W_2;L)$ such that $v|_{W_1}
  \not=0$ and $v$ is not in the image of $\sect(W_3;L) \to \sect(W_2;L)$.
\end{itemize}
\end{corollary}
\begin{proof}
By Corollary~\ref{cor:quantconic} there exist a conic closed Lagrangian
submanifold $\Lambda \subset\dT^*(V^2 \times \R)$ and
$K \in \Derb(\cor_{V^2 \times \R})$ such that
\begin{itemize}
\item [(a)] $\rho_{V^2}\cl T^*V^2 \times \dT^*\R \to T^*V^2$ induces a
  diffeomorphism between $(\Lambda \cap (T^*V^2 \times \dT^*\R))/\R^\times$ and
  $\Lambda_\varphi$,
\item [(b)] $\dot\SSi(K) = \Lambda$ and $K$ is simple along $\Lambda$.
\end{itemize}
We set $c_1 = (3Ar)^{-1}$ and $c_2 = (2Ar)^{-1}$.  We apply
Proposition~\ref{prop:sections-fixed-open} with $c_1,c_2,r$ and $V'=V^2$.
We obtain open subsets $W_i$ of $B_r^{V^2} \times \R$, $i=1,\dots,4$.
The hypothesis~\eqref{eq:hyp-part-quant}-(ii) implies that $K$
satisfies~\eqref{eq:hypSSF_split} and the
hypothesis~\eqref{eq:hyp-part-quant}-(iii) implies that $K$ satisfies~(iii) of
Proposition~\ref{prop:sections-fixed-open}.  Hence the proposition gives
$L\in \Mod(\cor_{W_4})$ satisfying the required properties.
\end{proof}

\section{Proof of the Gromov-Eliashberg theorem}
\label{sec:GET}

We use the notations in the statement of Theorem~\ref{thm:GE}.

\subsection{}
Up to a translation in $E$ it is enough to prove that $d\varphi_\infty|_0$ is a
symplectic linear map and we work near $0$.  We can also assume that
$\varphi_\infty(0)=0$.  By Lemma~\ref{lem:gen_pos} we can identify $E$ with
$T^*V$ for some vector space $V$ and assume that $\varphi_\infty$
satisfies~\eqref{eq:gen_pos_ineq} and~\eqref{eq:gen_pos_ineq2} for some
$A,r_0>0$.

\subsection{}
By Proposition~\ref{prop:appr_Hamisot}, for each $n\in \N$ we can find
a $\Cinf$ Hamiltonian isotopy $\Phi_n \cl E\times \R \to E$ and a compact
subset $C_n \subset E$ such that
$\| \varphi_n - \Phi_{n,1} \|_{\ol{B_{r_0}^E}} \leq 1/n$
and $\Phi_{n,t} |_{E\setminus C_n} = \id _{E\setminus C_n}$ for all $t\in\R$.  We
still have $\| \Phi_{n,1} - \varphi_\infty \|_{\ol{B_{r_0}^E}} \to 0 $ when
$n\to \infty$ and we may assume that $\varphi_n = \Phi_{n,1}$.
We choose $0<r<r_0$ (arbitrarily small).  By Lemma~\ref{lem:proche_gen_pos}
there exists $N_r \in \N$ such that
\begin{equation}\label{eq:gen_pos_ineq3}
\Lambda_{\varphi_n} \cap (B_{r}^{V^2} \times B_{3Ar}^{V^{*2}})
\subset  B_{r}^{V^2} \times B_{2Ar}^{V^{*2}}  ,
\quad \text{ for all $n \geq N_r$.}
\end{equation}

\subsection{}\label{sec:pf3}
We set $\Lambda^1_n = \Lambda_{\varphi_n} \cap (B_{r}^{V^2} \times B_{3Ar}^{V^{*2}})$.
By Proposition~\ref{prop:degree} the map $\Lambda^1_n \to B_{r}^{V^2}$ has degree
$1$, for $n\geq N_r$.
Let $W_i$, $i=1,\ldots,4$, be the non empty connected open subsets of
$B_{r}^{V^2}\times\R$ given by Corollary~\ref{cor:part-quant}.
We apply Corollary~\ref{cor:part-quant} to $\varphi_n$ and we
obtain $L_n\in \Mod(\cor_{W_4})$ such that
\begin{itemize}
\item [(i)] $\dot\SSi(L_n) \subset T^*W_4 \cap T^*_{\sigma>0}(V^2\times\R)$
and $\rho_{V^2}(\dot\SSi(L_n)) \subset  \Conv(\Lambda_n^1)$,
\item [(ii-a)] there exists $u_n\in \sect(W_4;L_n)$ such that
  $u_n|_{W_3} \not=0$ and $u_n|_{W_2} =0$,
\item [(ii-b)] or there exists $v_n\in \sect(W_2;L_n)$ such that $v_n|_{W_1}
  \not=0$ and $v_n$ is not in the image of $\sect(W_3;L_n) \to \sect(W_2;L_n)$.
\end{itemize}
One of the two possibilities (ii-a) or (ii-b) occurs infinitely many times.  Up
to taking a subsequence we will assume that (ii-a) holds for all $n\in \N$ (the
other case being similar).

\subsection{}\label{sec:pf4}
We define $L_\infty \in \Mod(\cor_{W_4})$ by the exact sequence
$$
0 \to \bigoplus_{n \geq N_r} L_n \to \prod_{n \geq N_r} L_n \to  L_\infty \to 0.
$$
The sections $u_n \in \sect(W_4; L_n)$ define $u\in \sect(W_4; L_\infty)$.
By~(ii-a) in~\S\ref{sec:pf3} we have $u|_{W_2} =0$.  Since $\ol{W_3}$ is compact
we have $\sect(\ol{W_3};\bigoplus_{n\geq N_r} L_n)
\simeq \bigoplus_{n\geq N_r}\sect(\ol{W_3}; L_n)$.
If $u=0$, we deduce that
$(u_n)_{n \geq N_r}|_{\ol{W_3}}$ belongs to
$\bigoplus_{n \geq N_r}\sect(\ol{W_3}; L_n)$,
which implies $u_n|_{\ol{W_3}} =0$ for $n$ big. But this
contradicts~(ii-a) and it follows that $u\not=0$. Hence $L_\infty$ is not
locally constant and $\dot\SSi(L_\infty) \not=\emptyset$.

\subsection{}
Proposition~\ref{prop:boundSSlim} gives $\dot\SSi(L_\infty) \subset 
\opb{\rho_{V^2}} (\bigcap_{k\geq N_r} \ol{\bigcup_{n\geq k} \Conv(\Lambda^1_n)})$.
By hypothesis $\{\Lambda^1_n\}_{n\in\N}$ converges to $\Lambda^1_\infty \eqdot 
\Lambda_{\varphi_\infty} \cap (B_{r}^{V^2} \times B_{3Ar}^{V^{*2}})$.
Since $\varphi_\infty$ satisfies~\eqref{eq:gen_pos_ineq2}, the set
$\Lambda^1_\infty$ is a section of the projection to the base
$T^*(B_{r}^{V^2}) \to B_{r}^{V^2}$.
Hence $\{\Conv(\Lambda^1_n)\}_{n\in\N}$ also converges to $\Lambda^1_\infty$
and we obtain
$\dot\SSi(L_\infty) \subset \opb{\rho_{V^2}} (\Lambda_{\varphi_\infty}) \cap
\opb{\pi_{V^2\times \R}}(W_4)$.

Let us choose $p\in \dot\SSi(L_\infty)$.  By the involutivity Theorem and
Proposition~\ref{prop:invol_invol} we obtain that $\Lambda_{\varphi_\infty}$ is
coisotropic at $\rho_{V^2}(p)$.
Since $\rho_{V^2}(p) \in B_{r}^{V^2} \times B_{3Ar}^{V^{*2}}$ and $r$ can be
chosen arbitrarily small, we deduce that $\Lambda_{\varphi_\infty}$ is
coisotropic at $0$, as required.

\providecommand{\bysame}{\leavevmode\hbox to3em{\hrulefill}\thinspace}

\vspace{2cm}
\noindent
\parbox[t]{21em}
{\scriptsize{
\noindent
St{\'e}phane Guillermou \\
Institut Fourier, Universit{\'e} Grenoble I, \\
email: Stephane.Guillermou@ujf-grenoble.fr
}}


\begin{thebibliography}{00}


\bibitem{E87} Ya.~M.~Eliashberg,
{\em A theorem on the structure of wave fronts and its application in 
symplectic topology,}
Funktsional. Anal. i Prilozhen. {\bf 21}, no. 3, p.~65--72, (1987).

\bibitem{Ga81} O.~Gabber,
{\em The integrability of the characteristic variety,}
Amer. J. Math. {\bf 103}, no. 3, p.~445--468 (1981).

\bibitem{Gr86} M.~Gromov,
{\em  Partial differential relations},
Ergebnisse der Mathematik und ihrer Grenzgebiete (3) {\bf 9}
Springer-Verlag (1986). 

\bibitem{GKS12} S.~Guillermou, M.~Kashiwara and P.~Schapira,
{\em Sheaf quantization of Hamiltonian isotopies and applications to
non displaceability problems,}
Duke Math. J. {\bf 161} no. 2, p.~201--245  (2012).

\bibitem{HLS13} V.~Humili{\`e}re, R.~Leclercq and S.~Seyfaddini,
{\em Coisotropic rigidity and $C^0$-symplectic geometry},
\/\texttt{arXiv:1305.1287}

\bibitem{KS82} M.~Kashiwara and P.~Schapira,
{\em Micro-support des faisceaux: applications aux modules diff{\'e}rentiels,}
C.~R.~Acad.\ Sci.\ Paris s{\'e}rie I Math {\bf 295} 8, p.~487--490 (1982).

\bibitem{KS85} \bysame,
{\em Microlocal study of sheaves,}
Ast{\'e}risque {\bf 128} Soc.\ Math.\ France (1985).


\bibitem{KS90} \bysame,
{\em Sheaves on Manifolds,}
\/Grundlehren der Math. Wiss. {\bf 292} Springer-Verlag (1990).

\bibitem{SKK} M.~Sato, T.~Kawai and M.~Kashiwara,
{\em Microfunctions and pseudo-differential equations,}
\/in Komatsu (ed.), {\em Hyperfunctions and pseudo-differential  equations.}
\/Proceedings Katata 1971, Lecture Notes in Math.
{\bf 287} p.~265--529 (1973)

\bibitem{T08} D.~Tamarkin, 
{\em Microlocal conditions for non-displaceability,}
\texttt{arXiv:0809.1584}

\bibitem{Vic12} N.~Vichery,
{\em Homological differential calculus,}
\texttt{arXiv:1310.4845}

\bibitem{V11} C.~Viterbo,
{\em An Introduction to Symplectic Topology through Sheaf theory,}
Lectures at Princeton Fall 2010 and New York Spring 2011, \\
\texttt{www.dma.ens.fr/\~{}viterbo}

\end{thebibliography}
\end{document}